\documentclass[11pt,%
               a4paper,%
               oneside,%
               reqno%
               ]{amsart}
                                   
\usepackage[T1]{fontenc}

\usepackage{lmodern}
\usepackage[utf8]{inputenc}	
\usepackage[english]{babel}

\usepackage[dvipsnames]{xcolor}

\usepackage{fullpage}

\usepackage{amssymb}

\usepackage{slashed}

\usepackage{tikz}
\usetikzlibrary{cd,decorations.pathmorphing,positioning,arrows,matrix,calc,backgrounds,decorations.markings}

\usepackage[hidelinks]{hyperref}

\newcommand{\numberset}{\mathbb} 
 
\newcommand{\Z}{\numberset{Z}} 
\newcommand{\Q}{\numberset{Q}} 
\newcommand{\R}{\numberset{R}}

\newcommand{\bC}{\numberset{C}}
\newcommand{\bH}{\mathbb{H}}
\newcommand{\bT}{\mathbb{T}}

\newcommand{\cG}{\mathcal G}

\newcommand{\cL}{\mathcal L}
\newcommand{\cR}{\mathcal R}

\DeclareMathOperator{\RK}{RK}
\newcommand{\KK}{\mathit{KK}}
\DeclareMathOperator{\Ind}{Ind}
\DeclareMathOperator{\supp}{supp}
\DeclareMathOperator{\Res}{Res}

\DeclareMathOperator{\Tor}{Tor}

\newcommand{\medwedge}{{\textstyle\bigwedge}}

\newcommand{\KU}{\mathrm{KU}}
\newcommand{\sfZ}{\underline{\mathbb Z}}

\newcommand{\absv}[1]{\left|#1\right|}
\newcommand{\nvGR}[1]{\left\rceil#1\right\lceil}

\newcommand{\id}{\mathrm{id}}

\usepackage[hyperpageref]{backref}
\usepackage[nobysame,alphabetic,initials]{amsrefs}

\DefineSimpleKey{bib}{how}
\DefineSimpleKey{bib}{mrclass}
\DefineSimpleKey{bib}{mrnumber}
\DefineSimpleKey{bib}{fjournal}
\DefineSimpleKey{bib}{mrreviewer}

\renewcommand{\PrintDOI}[1]{%
  \href{http://dx.doi.org/#1}{{\tt DOI:#1}}%
}
\renewcommand{\eprint}[1]{#1}

\BibSpec{book}{%
    +{}  {\PrintPrimary}                {transition}
    +{.} { \PrintDate}                  {date}
    +{.} { \textit}                     {title}
    +{.} { }                            {part}
    +{:} { \textit}                     {subtitle}
    +{,} { \PrintEdition}               {edition}
    +{}  { \PrintEditorsB}              {editor}
    +{,} { \PrintTranslatorsC}          {translator}
    +{,} { \PrintContributions}         {contribution}
    +{,} { }                            {series}
    +{,} { \voltext}                    {volume}
    +{,} { }                            {publisher}
    +{,} { }                            {organization}
    +{,} { }                            {address}
    +{,} { }                            {status}
    +{,} { \PrintDOI}                   {doi}
    +{,} { \PrintISBNs}                 {isbn}
    +{}  { \parenthesize}               {language}
    +{}  { \PrintTranslation}           {translation}
    +{;} { \PrintReprint}               {reprint}
    +{.} { }                            {note}
    +{.} {}                             {transition}
}
\BibSpec{article}{%
    +{}  {\PrintAuthors}                {author}
    +{,} { \textit}                     {title}
    +{.} { }                            {part}
    +{:} { \textit}                     {subtitle}
    +{,} { \PrintContributions}         {contribution}
    +{.} { \PrintPartials}              {partial}
    +{,} { }                            {journal}
    +{}  { \textbf}                     {volume}
    +{}  { \PrintDatePV}                {date}
    +{,} { \issuetext}                  {number}
    +{,} { \eprintpages}                {pages}
    +{,} { }                            {status}
    +{,} { \PrintDOI}                   {doi}
    +{,} { \eprint}        {eprint}
    +{}  { \parenthesize}               {language}
    +{}  { \PrintTranslation}           {translation}
    +{;} { \PrintReprint}               {reprint}
    +{.} { }                            {note}
    +{.} {}                             {transition}
}
\BibSpec{collection.article}{%
    +{}  {\PrintAuthors}                {author}
    +{,} { \textit}                     {title}
    +{.} { }                            {part}
    +{:} { \textit}                     {subtitle}
    +{,} { \PrintContributions}         {contribution}
    +{,} { \PrintConference}            {conference}
    +{}  {\PrintBook}                   {book}
    +{,} { }                            {booktitle}
    +{,} { \PrintDateB}                 {date}
    +{,} { pp.~}                        {pages}
    +{,} { }                            {publisher}
    +{,} { }                            {organization}
    +{,} { }                            {address}
    +{,} { }                            {status}
    +{,} { \PrintDOI}                   {doi}
    +{,} { \eprint}        {eprint}
    +{}  { \parenthesize}               {language}
    +{}  { \PrintTranslation}           {translation}
    +{;} { \PrintReprint}               {reprint}
    +{.} { }                            {note}
    +{.} {}                             {transition}
}
\BibSpec{misc}{%
  +{}{\PrintAuthors}  {author}
  +{,}{ \textit}      {title}
  +{.}{ }             {how}
  +{}{ \parenthesize} {date}
  +{,} { available at \eprint}        {eprint}
  +{,}{ available at \url}{url}
  +{,}{ }             {note}
  +{.}{}              {transition}
}

\usepackage{amsthm}
\theoremstyle{plain}
\newtheorem{theorem}{Theorem}[section]
\newtheorem{proposition}[theorem]{Proposition}
\newtheorem{lemma}[theorem]{Lemma}
\newtheorem{corollary}[theorem]{Corollary}

\newtheorem{theoremA}{Theorem}

\theoremstyle{definition} 
\newtheorem{definition}[theorem]{Definition}

\theoremstyle{remark}
\newtheorem{assumption}[theorem]{Assumption}
\newtheorem{example}[theorem]{Example}
\newtheorem{remark}[theorem]{Remark}

\author{Valerio Proietti}
\address{Department of Mathematics, University of Oslo, P.O. box 1053, Blindern, 0316 Oslo, Norway}
\email{valeriop@math.uio.no}

\author{Makoto Yamashita}
\address{Department of Mathematics, University of Oslo, P.O. box 1053, Blindern, 0316 Oslo, Norway}
\email{makotoy@math.uio.no}

\title[Homology and K-theory of dynamical systems, III.]{Homology and \texorpdfstring{$K$}{K}-theory of dynamical systems\\ III. Beyond stably disconnected Smale spaces}
\date{v4: November 11, 2024; v3: October 17, 2023; v2: October 28, 2022; v1: July 7, 2022.}
\DeclareRobustCommand{\SkipTocEntry}[5]{}

\begin{document}

\begin{abstract}
We study homological invariants of étale groupoids arising from Smale spaces, continuing our previous work, but going beyond the stably disconnected case by incorporating resolutions in the space direction.
We show that the homology groups defined by Putnam are isomorphic to the Crainic--Moerdijk groupoid homology with integer coefficients.

We also show that the $K$-groups of C$^*$-algebras of stable and unstable equivalence relations have finite rank.
For unstably disconnected Smale spaces, we provide a cohomological spectral sequence whose second page shows Putnam's (stable) homology groups, and converges to the $K$-groups of the unstable C$^*$-algebra.
\end{abstract}

\makeatletter
\@namedef{subjclassname@2020}{\textup{2020} Mathematics Subject Classification}
\makeatother

\subjclass[2020]{37B02; 55T25, 46L80}
\keywords{groupoid homology, topological dynamics, Smale spaces, spectral sequences, $K$-theory}

\maketitle
\setcounter{tocdepth}{1}
\tableofcontents


\addtocontents{toc}{\SkipTocEntry}\section*{Introduction}

In this paper we study homological and $K$-theoretical invariants of a class of hyperbolic dynamical systems known as \emph{Smale spaces} introduced by D.~Ruelle~\cite{ruelle:thermo}, continuing our work~\citelist{\cite{valmak:groupoid}\cite{valmak:groupoidtwo}}.
Previously we have focused on stably disconnected systems, which are represented by \emph{ample groupoids}, i.e., totally disconnected étale topological groupoids.
Now we lift this restriction and work with general non-wandering Smale spaces. 

Smale spaces capture hyperbolicity phenomena in a topological setting, akin to the dynamics of Anosov diffeomorphisms, and more generally, basic sets of Axiom A diffeomorphisms, in the differentiable (smooth) setting~\cite{bs:dynbook}. Other examples of Smale spaces include shifts of finite type, generalized solenoids, systems arising from self-similar group actions~\citelist{\cite{MR2162164}\cite{MR2526786}}, substitution tiling systems~\cite{MR1631708}, etc.

Aside from the purely dynamical interest in these systems, there is a fruitful interplay with the theory of operator algebras, via the construction of groupoid C$^*$-algebras. Given a second countable locally compact groupoid $G$ with a continuous Haar system, the convolution product on the space of compactly supported continuous functions gives rise to a complex algebra with involution that can be completed to a C$^*$-algebra~\cite{MR584266}.
This construction can be further generalized to incorporate continuous actions of $G$ on C$^*$-algebras, and can be analyzed by homological invariants such as the operator $K$-groups.

For Smale spaces, the C$^*$-algebras constructed from the stable and unstable equivalence relation capture interesting aspects of the homoclinic and heteroclinic structure of expansive dynamics~\citelist{\cite{put:algSmale}\cite{thomsen:smale}\cite{mats:ruellemarkov}}.
The C$^*$-algebras attached to Smale spaces provide a unified framework that contain (AF-algebras inside) the Cuntz--Krieger algebras, the Bunce--Deddens algebras, irrational rotation algebras, among others.
They are also known to have good structural properties one would expect in the theory of classification of C$^*$-algebras, see~\citelist{\cite{dgy:rrzero}\cite{deeley:strung}\cite{kp:clas}\cite{phil:clas}} and references therein. 

Another intriguing homological invariant for dynamical system is the homology of étale groupoids with finite cohomological dimension due to M.~Crainic and I.~Moerdijk, which is based on sheaves, derived formalism, and simplicial methods~\cite{cramo:hom}.
For \emph{ample} groupoids, which correspond to dynamical systems on totally disconnected spaces, a connection between this homology and the operator $K$-theory of the associated C$^*$-algebra was recently popularized by H.~Matui~\citelist{\cite{MR2876963}\cite{MR3552533}}, and gained a lot of attention in the C$^*$-algebra community.
More precisely, Matui has conjectured that groupoid homology (suitably periodicized) and $K$-groups of the groupoid C$^*$-algebra are isomorphic.
Although the original form of this conjecture is shown to be incorrect by some counterexamples~\citelist{\cite{scarparo:hk}\cite{deeley:hk}}, nonetheless it is still confirmed in many interesting cases by concrete computations~\citelist{\cite{simsfarsi:hk}\cite{bdgw:matui}}.

In our previous works~\citelist{\cite{valmak:groupoid}\cite{valmak:groupoidtwo}} we provided a more systematic picture behind this correspondence, constructing a homological spectral sequence whose $E^2$-sheet is given by the homology groups of ample groupoids $G$, and abuts to the $K$-groups of certain $G$-C$^*$-algebras~\cite{valmak:groupoid}, based on the Meyer--Nest theory of triangulated categorical structure on equivariant $\KK$-theory~\citelist{\cite{meyernest:tri}\cite{meyer:tri}}.
When the groupoid has \emph{torsion-free} stabilizers and satisfies (a stronger form of) the Baum--Connes conjecture as studied by J.-L.~Tu~\cite{MR1703305}, the abutment is indeed the $K$-groups of the groupoid C$^*$-algebra of $G$.

Turning back to Smale spaces, I.~F.~Putnam developed a theory of homology based on an intricate resolution of Smale spaces by symbolic dynamics (\emph{shifts of finite type}), see~\citelist{\cite{put:HoSmale}\cite{dkw:func}\cite{val:smale}}.
This theory can be seen as an attempt to classify Axiom A systems by means of purely combinatorial data, an insight that goes back to S.~Smale himself, in parallel with the theory of Morse--Smale systems.

More precisely, the main ingredient in Putnam's homology is a pair of Smale spaces $(Y, \psi)$ and $(Z, \zeta)$ that map onto $(X, \phi)$ such that: $(Y, \psi)$ has totally disconnected unstable sets, and the factor map induces a homeomorphism of stable sets (an \emph{$s$-bijective map}); $(Z, \zeta)$ has totally disconnected stable sets, and the factor map induces a homeomorphism of unstable sets (a \emph{$u$-bijective map}).
Then the iterated fiber products of these spaces give a simplicial-cosimplicial system of shifts of finite type, leading to this homology.

Putnam posed a number of interesting questions which have been a source of inspiration for this work, and in particular predicted a spectral sequence relating his homology groups to the operator $K$-groups of the groupoid~\cite{put:HoSmale}*{Question 8.4.1}.

In~\cite{valmak:groupoidtwo}, for the non-wandering Smale spaces \emph{with totally disconnected stable sets}, we showed that Putnam's stable homology coincides with the Crainic--Moerdijk homology for the unstable equivalence relation groupoid, reduced to a transversal, with integer coefficients.
Combined with the above spectral sequence from triangulated categorical structure of equivariant $\KK$-theory, it gave an answer to Putnam's question in this class of Smale spaces.

Motivated by this, in this paper we further generalize the identification of the homology groups of non-wandering Smale spaces, as follows.

\begin{theoremA}[Theorem~\ref{thm:compar-homologies}]\label{thmA:homology-compar}
Let $(X,\phi)$ be a non-wandering Smale space, and $T$ be a transversal for the unstable equivalence relation groupoid $R^u(X,\phi)$.
Then Putnam's stable homology $H^s_\bullet(X,\phi)$ is isomorphic to the groupoid homology of the étale groupoid $G = R^u(X,\phi)|_T$ with coefficients in the constant sheaf of integers:
\[
H^s_k(X,\phi)\cong H_k(G,\sfZ).
\]
\end{theoremA}

One crucial ingredient is a cosimplicial system of $G$-sheaves resolving the constant sheaf of integers $\sfZ$, which naturally arises from a cosimplicial system of Smale spaces over $(X, \phi)$. This system is associated to the collection of iterated fiber products of the space $(Z, \zeta)$ above. We combine sheaf theoretic methods for the associated complex together with techniques behind our previous work to obtain the above theorem.

The above result, in combination with results from~\cite{valmak:kpd}, leads to a number of concrete structural results on homology of Smale spaces.
On one hand, we obtain a Künneth type formula for the homology of Smale spaces (Theorem~\ref{thm:kunses}), generalizing results in~\cite{valmak:groupoidtwo} and~\cite{dkw:dyn} to arbitrary non-wandering Smale spaces.
On the other, we give a solution to a question by Putnam~\cite{put:HoSmale}*{Question 8.3.2} on a relation between his homology of $(X, \phi)$ and the cohomology of the underlying space $X$ (see Theorem~\ref{thm:poindualsmale} for the precise statement).

\medskip
We next turn to the problem of comparison between homology and $K$-theory.
As we remarked above, there is a homological spectral sequence whose $E^2$-sheet is the Putnam homology, and which abuts to the $K$-groups of $C^* R^u(X, \phi)$ for Smale spaces with totally disconnected stable sets.
Here we give the dual analogue of this result, for Smale spaces whose \emph{unstable} sets are totally disconnected.

\begin{theoremA}[Corollary~\ref{cor:spec-seq-Smale-tot-disconn-stable-sets}]\label{thm:b}
Let $(X, \phi)$ be a non-wandering Smale space with totally disconnected unstable sets.
There is a cohomological spectral sequence abutting to $K_\bullet(C^* R^u(X, \phi))$, with $E_2$-sheet given by $H^s_\bullet(X, \phi)$.
\end{theoremA}

Again this gives an answer to~\cite{put:HoSmale}*{Question 8.4.1} in this class of Smale spaces.
While the claim is quite analogous to our previous result, the proof turns out to be somewhat different.

For the case of stable sets being totally disconnected, one starts with a factor map from a shift of finite type that maps onto $X$ by an $s$-bijective map.
This leads to an open subgroupoid of the étale groupoid $G$, where the theory of projective resolutions in triangulated categories~\citelist{\cite{meyernest:tri}\cite{meyer:tri}} is applicable through restriction and induction functors between the corresponding equivariant $\KK$-categories.

For the case of unstable sets being totally disconnected, one starts with a factor map from a shift of finite type by a $u$-bijective map.
This leads to a totally disconnected space $T_0$ on which the étale groupoid $G$ acts.
Moreover, we obtain a simplicial system $T_\bullet$ of totally disconnected $G$-spaces, with a compatible action of the symmetric group, by taking the iterated pullback construction of the structure map $T_0 \to T$.
Morally speaking, this represents a kind of injective resolution, which however does not behave as nicely as projective resolutions in the framework of $\KK$-categories.

Here we instead adapt ideas behind G.~Segal's work~\cite{MR232393} on simplicial spaces and spectral sequences.
Using a variation of geometric realization suitable in the setting of locally compact spaces, we obtain a spectral sequence analogous to the Atiyah--Hirzebruch spectral sequence but without CW-complexes around.
The abutment to the desired $K$-groups, which corresponds to the Baum--Connes conjecture in the case of projective resolutions, comes from the sheaf theoretic comparison of cohomological invariants for the nerves of the associated transformation groupoid on the $G$-simplicial spaces $T_\bullet$.

\medskip
Finally, in the setting of general non-wandering Smale spaces, while we cannot directly relate the homology groups to the $K$-groups, the techniques underlying Theorem~\ref{thm:b} still allow us to prove the following finiteness result for the $K$-groups of the stable and unstable equivalence relation groupoids for Smale spaces.

\begin{theoremA}[Theorem~\ref{thm:ftrk}]\label{thmA:ftrk}
Let $(X,\phi)$ be a non-wandering Smale space.
Then the $K$-groups 
\[
K_\bullet(C^*R^s(X,\phi)),\qquad K_\bullet(C^*R^u(X,\phi))
\]
of the associated groupoid C$^*$-algebras are of finite rank.
\end{theoremA}

This was conjectured in~\cite{MR3692021}, and has implications on the structure of the \emph{Ruelle algebras}, i.e., the crossed products of $C^*R^s(X,\phi)$ and $C^*R^u(X,\phi)$ by the integer actions induced by $\phi$ (see Remark~\ref{rem:kpwiso}).
It also shows that C$^*$-algebras of Smale spaces do \emph{not} exhaust the class of (simple) classifiable, real rank zero C$^*$-algebras, which is a natural question in the context of the classification program of nuclear C$^*$-algebras, see~\citelist{\cite{deeley:strung}\cite{dgy:rrzero}} for further details.

\medskip
The paper is organized as follows.
In Section~\ref{sec:prelim} we recall the basic notions and fix our conventions used throughout the paper.
We also briefly summarize our previous work to set the background for this work.

Section~\ref{sec:symm-simp-sp} is the technical core of the paper, where we use methods from simplicial homotopy theory and manipulate sheaf resolutions to achieve approximation results in cohomology and $K$-theory.
These results are useful in view of the (co)simplicial structure arising from factor maps of shifts of finite type.

In the next three sections we consider applications to Smale spaces.
The main results here are the comparison theorem between Putnam's homology and groupoid homology (Theorem~\ref{thmA:homology-compar}) in Section~\ref{sec:comp-hom}, the spectral sequence for \emph{unstably} disconnected Smale spaces relating homology and $K$-groups (Theorem~\ref{thm:b}) in Section~\ref{sec:smale-sp-tot-disc-unst-sets}, and finally the finiteness result on the $K$-groups of Smale spaces (Theorem~\ref{thmA:ftrk}) in Section~\ref{sec:finite-gen-smale-sp}.

\addtocontents{toc}{\SkipTocEntry}\subsection*{Acknowledgments}

V.P.:~this research was supported by: Foreign Young Talents' grant (National Natural Science Foundation of China, QN2021137002L), CREST Grant Number JPMJCR19T2 (Japan Science and Technology Agency of the Ministry of Education,
Culture, Sports, Science and Technology), Marie Skłodowska-Curie Individual Fellowship (Horizon Europe, European Commission, Project No. 101063362).

M.Y.:~this research was funded, in part, by The Research Council of Norway [project 300837].

\section{Preliminaries}\label{sec:prelim}

We fix conventions in use throughout the paper.
We only briefly recall definitions and mostly refer to and generally follow the treatment in~\citelist{\cite{valmak:groupoid}\cite{valmak:groupoidtwo}}.

\subsection{Topological groupoids}

We mainly work with second countable, locally compact, Hausdorff groupoids.
Given such a groupoid $G$, we denote its base space by $G^{(0)}$, with structure maps $s, r \colon G \to G^{(0)}$, and the $n$-th
nerve space (for $n\geq 1$) given by
\[
G^{(n)} = \{(g_1, \dots, g_n) \in G^n \colon s(g_i) = r(g_{i+1}) \}.
\]
We say that $G$ is \emph{étale} if $s$ and $r$ are local homeomorphisms, and \emph{ample} if it is étale and its base space is totally disconnected.

When we work with operator algebras, we further assume that $G$ admits a \emph{continuous Haar system}~\cite{MR584266}.
A choice of continuous left Haar system is denoted by $(\lambda^x)_{x \in G^{(0)}}$, and the associated C$^*$-algebras are denoted by $C^*(G, \lambda)$ and $C^*_r(G, \lambda)$, for the universal and reduced model, respectively.
Most concrete groupoids we consider will be amenable so that this distinction disappears.
When the choice of $\lambda$ is already understood, we also write $C^* G$ and $C^*_r G$.
When $G$ is étale, we take $\lambda$ to be the counting measure.

Following~\cite{cramo:hom}, we say that a (topological) groupoid homomorphism $f \colon H \to G$ is a \emph{Morita equivalence} if
\[
\{ (x, g) \mid x \in H^{(0)}, g \in G, f(x) = r(g) \} \to G^{(0)}, \quad (x, g) \mapsto s(g)
\]
is an étale surjection, and
\[
\begin{tikzcd}
 H \arrow[r, "f"] \arrow[d, "{(r,s)}"] & G \arrow[d, "{(r,s)}"]\\
 H^{(0)} \times H^{(0)} \arrow[r, "{(f,f)}"] & G^{(0)} \times G^{(0)}
\end{tikzcd}
\]
is a pullback diagram.
Two groupoids $G$, $H$ are \emph{Morita equivalent} if there is another groupoid $K$ and Morita equivalence homomorphisms $K \to G$, $K \to H$.
This is equivalent to a different notion, probably more familiar to operator algebraists, given through the existence of a \emph{principal bibundle} $H \curvearrowright X \curvearrowleft G$, as defined in~\cite{murewi:morita} (see~\cite{simsfarsi:hk}*{Prop.~3.10} for the equivalence).

\begin{assumption}\label{ass:tr}
      In this paper, when $G$ is a groupoid which is not étale, we assume the existence of an abstract transversal $T\subseteq G^{(0)}$ such that $G|_T$ can be given an étale topology, and there is a Morita equivalence $G\sim G|_T$; the details on how to re-topologize $G|_T$ and how to establish the Morita equivalence with $G$ are in~\cite{put:spiel}. This assumption on the existence of $T$ is satisfied in most examples arising from geometric or dynamical situations.
\end{assumption}

Given groupoid homomorphisms $f_i \colon H_i \to G$ for $i = 1, \ldots, n$, we consider \emph{multiple groupoid pullbacks} $H_1 \times_G \dots \times_G H_n$ as we did in~\cite{valmak:groupoid}.
This is a groupoid whose base is given by
\[
\{ (x_1, g_1, x_2, \dots, g_{n-1}, x_n) \mid x_i \in (H_i)^{(0)}, g_i \in G^{f_i(x_i)}_{f_{i+1}(x_{i+1})} \},
\]
and an arrow from $(x_1, g_1, \dots, x_n)$ to $(x'_1, g'_1, \dots, x'_n)$ is given by a tuple $(h_1, \dots, h_n)$ for $h_i \in (H_i)^{x'_i}_{x_i}$ such that the diagram
\[
\begin{tikzcd}
f_1(x'_1) & f_2(x'_2) \arrow[l, "g'_1"] & \dots & f_n(x'_n) \arrow[l, "g'_{n-1}"] \\
f_1(x_1) \arrow[u, "f_1(h_1)"] & f_2(x_2) \arrow[l, "g_1"] \arrow[u, "f_2(h_2)"] & \dots & f_n(x_n) \arrow[l, "g_{n-1}"] \arrow[u, "f_n(h_n)"]
\end{tikzcd}
\]
is commutative.

\begin{proposition}\label{prop:triv-mult-pllback}
Let $G$ be a topological groupoid, and $n$ be a nonnegative integer.
Then the fibered product groupoid $G^{\times_G (n+1)}$, with respect to the copies of identity homomorphisms $G \to G$, is Morita equivalent to $G$.
\end{proposition}

\begin{proof}
We give explicit Morita equivalence homomorphisms between $G$ and $G^{\times_G (n+1)}$ in both ways (formally, we only need to know one).
In one direction, $f \colon G^{\times_G (n+1)} \to G$ is given by
\[
(G^{\times_G (n+1)})^{(0)} \to G^{(0)}, \quad (x_0, g_1, \dots, x_n) \mapsto x_n
\]
at the level of base, and
\[
G^{\times_G (n+1)} \to G, \quad (g_0, \dots, g_n) \mapsto g_n
\]
at the level of arrows.
In the other direction, $f' \colon G^{\times_G (n+1)} \to G$ is given by
\[
G^{(0)} \to (G^{\times_G (n+1)})^{(0)}, \quad x \mapsto (x, \id_x, \dots, x)
\]
at the level of base, and
\[
G \to G^{\times_G (n+1)}, \quad g\mapsto (g, g, \dots, g)
\]
at the level of arrows.
\end{proof}

\subsection{Homology of étale groupoids}

In this section let us recall the homology of étale groupoids with coefficients in equivariant sheaves as defined by Crainic and Moerdijk~\cite{cramo:hom}. 

A \emph{$G$-sheaf} is a sheaf $F$ on $G^{(0)}$ equipped with a continuous action of $G$, concretely given by a map of sheaves $s^* F \to r^* F$ on $G$.
A $G$-sheaf $F$ is called \emph{(c-)soft} when the underlying sheaf on $G^{(0)}$ has that property, i.e, for any (compact) closed subset $S \subset G^{(0)}$ and any section $x \in \Gamma(S, F)$, there is an extension $\tilde x \in \Gamma(X, F)$.

\begin{assumption}
      We are going to make the standing assumption that the cohomological dimension of any open sets of $G^{(0)}$ is bounded by some fixed integer $N$: to be precise, $H_c^k(U, F) = 0$ for any open subset $U \subset G^{(0)}$, any sheaf of commutative groups $F$, and any $k > N$.
\end{assumption}

The assumption above is guaranteed when $G^{(0)}$ is a metrizable space of Lebesgue topological dimension $N$, which covers all the concrete examples we consider.
To see this, observe that when the topological dimension of any compact subset $A \subset G^{(0)}$ is bounded by $N$~\cite{MR0482697}*{Section 3.1}, then the Čech cohomology $\check{H}^\bullet(A; F)$ will vanish in degree above $N$.
Then, by paracompactness, $\check{H}^\bullet(A; F)$ agrees with the sheaf cohomology $H^\bullet(A; F) = R^\bullet \Gamma_A(F)$ for the right derived functor of the functor $\Gamma_A(F) = \Gamma(A, F)$, and we can combine~\cite{MR0345092}*{Remarque II.4.14.1 and Théorème II.4.15.1} to get the claim.

Let $F$ be a $G$-sheaf of commutative groups.
Then there is a resolution of $F$ as above by c-soft $G$-sheaves, and the \emph{homology with coefficient $F$}, denoted $H_\bullet(G, F)$, is defined as the homology of the total complex of the double complex $(C_i^j)_{0 \le i,j}$, whose terms are given by
\[
C_i^j = \Gamma_c(G^{(i)}, s^* F^j),
\]
which has homological degree $i - j$.
In more generality, when $F_\bullet$ is a homological complex of $G$-sheaves bounded from below, consider a resolution of each $F_j$ by c-soft $G$-sheaves $F_j^k$ as above, which fit into a bicomplex resolving $F_\bullet$.
Then the \emph{hyperhomology with coefficient $F_\bullet$}, denoted by $\bH_\bullet(G, F_\bullet)$, is the homology of triple complex with terms
\[
C_{i,j}^k = \Gamma_c(G^{(i)}, s^* F_j^k),
\]
which has homological degree $i + j - k$.

When two étale groupoids $G$ and $H$ are Morita equivalent, there are natural correspondences between the $G$-sheaves and $H$-sheaves inducing an isomorphism of groupoid homology.
In particular, if $f \colon H \to G$ is a Morita equivalence homomorphism, we have
\[
\bH_\bullet(G, F_\bullet) \cong \bH_\bullet(H, f^* F_\bullet)
\]
for any complex $F_\bullet$ of $G$-sheaves as above.

Recall that a groupoid is called \emph{elementary} if it is the disjoint union of groupoids $K_i$, each of which is isomorphic to the product of a second countable locally compact space and a transitive principal groupoid on a finite set.
An ample groupoid is an \emph{AF groupoid} when it is an increasing union of elementary groupoids~\cite{MR584266}*{Definition III.1.1}.

\begin{proposition}[cf.~\cite{MR2876963}*{Theorem 4.11}]\label{prop:af-grpd-triv-homol}
Let $G$ be an AF groupoid, and $F$ be a $G$-sheaf.
Then $H_p(G, F) = 0$ for $p > 0$.
\end{proposition}

\begin{proof}
The proof is a close analogue of that for~\cite{MR2876963}*{Theorem 4.11}.
For each $i$, the groupoid $K_i$ is Morita equivalent to the space $K_i \backslash K_i^{(0)}$ regarded as a trivial groupoid.
The latter space is second countable and totally disconnected, hence has trivial groupoid homology (see~\cite{valmak:groupoid}*{Proposition 1.8}).
Consider a class of $H_p(G, F)$, represented by a cycle $f \in \Gamma_c(G^{(p)}, s^* F)$.
Since $f$ has compact support, there is an index $i$ with $f \in \Gamma_c(K_i^{(p)}, s^* F)$.
Then we have $H_p(K_i, F) = 0$ by the Morita invariance of groupoid homology, thus $f$ must be a boundary.
\end{proof}

\subsection{Derived functor formalism}
Let us briefly recall the formalism of derived functors in groupoid homology from~\cite{cramo:hom}*{Section 4}.
Let $G$ and $G'$ be étale groupoids, and $\phi\colon G \to G'$ be a continuous groupoid homomorphism.
Then, for each $x \in G^{\prime(0)}$, the \emph{comma groupoid} $x / \phi$ is defined as the groupoid whose objects are the pairs $(y, g')$, where $y \in G^{(0)}$ and $g' \in G_{x}^{\prime \phi(y)}$, and an arrow from $(y_1, g'_1)$ to $(y_2, g'_2)$ is given by $g \in G_{y_1}^{y_2}$ such that $\phi(g) g'_1 = g'_2$.
This is an étale groupoid equipped with a homomorphism $\pi_x \colon x / \phi \to G$.

When $F$ is a $G$-sheaf of commutative groups, we can consider a simplicial system of $G'$-sheaves, denoted by $B_\bullet(\phi, F)$, which at the level of stalks is given by
\[
B_n(\phi, F)_x = \Gamma_c((x/\phi)^{(n)}, s^* \pi_x^* F).
\]
The \emph{left derived functor} $\cL \phi_! F_\bullet$ for a homological complex of $G$-sheaves bounded from below $F_\bullet$ is
represented by the total complex of the triple complex of $G'$-sheaves with terms $B_i(\phi, F_j^k)$ with homological degree $i + j - k$, where $F_j^\bullet$ is a resolution of $F_j$ by c-soft $G$-sheaves.
Note that this is well defined up to a quasi-isomorphism of $G'$-sheaves.
The \emph{$n$-th derived functor}, denoted by $L_n \phi_! F_\bullet$, is the $G'$-sheaf given as the $n$-th homology of $\cL \phi_! F_\bullet$.
The fiber of this sheaf is given by~\cite{cramo:hom}*{Proposition 4.3}:
\begin{equation}\label{eq:n-th-derived-stalk}
(L_n \phi_! F_\bullet)_x = \bH_n( x / \phi , \pi_x^* F_\bullet).
\end{equation}
If $G'$ is the trivial groupoid and $\phi$ is the unique homomorphism $G \to G'$, then we recover the definition of $\bH_n(G, F_\bullet)$.

Aside from the pullback functor (also called inverse image functor) for sheaves, we will also need the \emph{direct image} functor, defined as $g_*F(U)=F(g^{-1}(U))$~\cite{MR0345092}.
The functor $g_*$ can be defined in the setting of groupoid-equivariant sheaves, and it is still right adjoint to the pullback functor (see~\cite{cramo:hom}*{Section 2.3} for details).
Let us note that, if $g$ is proper, then $g_*$ coincides with the functor $g_!$, called direct image functor with compact supports. In the familiar setting of topological spaces, $g_!$ is left adjoint to $g^*$ whenever the map $g$ is étale, however we emphasize that, in the case of groupoid actions considered here, the corresponding statement requires additional hypotheses~\cite{cramo:hom}*{Remark 5.2}.

\subsection{Triangulated category from groupoid equivariant \texorpdfstring{$\KK$}{KK}-theory}
\label{sec:triang-cat-from-groupoid-equiv}

Let us briefly summarize the contents of
\citelist{\cite{valmak:groupoid}\cite{valmak:groupoidtwo}} in order to provide the background for this paper.
We consider the equivariant Kasparov category $\KK^G$ for a second countable locally compact Hausdorff groupoid $G$ (in our case $G$ can be further assumed to be étale)~\cite{gall:kk}. The crucial point is that this category is triangulated, 
which allows us to perform homological algebra constructions following the blueprint in~\cite{meyernest:tri}.

Let $H$ be an open subgroupoid of $G$ satisfying $H^{(0)}=G^{(0)}$.
Then we can define the induction functor $\Ind_H^G\colon \KK^H \to \KK^G$ as
\[
\Ind_H^G B = (C_0(G) \otimes_{C_0(G^{(0)})} B) \rtimes H,
\]
where $\otimes_{C_0(G^{(0)})}$ denotes the C$^*$-algebraic balanced tensor product over $C_0(G^{0})$.
This functor is left adjoint to the natural restriction functor $\Res_H^G\colon \KK^G \to \KK^H$.

Let us set $L=\Ind_H^G\Res_H^G$ and $P_n = L^{n+1} A$ for $n \ge 0$. The adjunction gives $L$ the structure of comonad, and by a standard argument $P_\bullet$ becomes an augmented simplicial object over $A$.
Applying the theory of projective resolutions in triangulated categories, as a special case of the ``ABC spectral sequence''~\cite{meyer:tri} for the functor $A \mapsto K_\bullet(G \ltimes A)$, we obtain a spectral sequence whose $E^2$-sheet shows the homology groups of the chain complex $K_\bullet(G\ltimes P_\bullet)$. The convergence is to $K_\bullet(G \ltimes P)$ for a naturally defined object $P$, which should be viewed as an approximation of $A$.

In broad terms, the object $P$ is constructed as follows: since the objects $P_n$ satisfy a certain projectivity condition, it is possible to iterate the construction of exact triangles
\begin{equation*}
	P_n \to N_n \to N_{n+1} \to P_n[1],
\end{equation*}
where we use the convention $A=N_0$ and $(-)[1]$ denotes suspension.
The $N_k$'s form an inductive system whose homotopy colimit we denote $N$.
Now the object $P$ is characterized as fitting a unique exact triangle
\begin{equation*}
	P\to A \to N \to P[1].
\end{equation*}

The assumption the stabilizers of $G$ are torsion-free implies that $K_\bullet(G\ltimes P)$ can be identified with the left-hand side of the Baum--Connes conjecture with coefficients in $A$, see~\cite{val:kthpgrp}.
The theorem below is a summary of the situation described so far.

\begin{theorem}[\cite{valmak:groupoid}*{Corollary 3.6}]\label{thm:spec-seq-gen}
Let $G$ be an ample groupoid with torsion-free stabilizers satisfying the Baum--Connes conjecture with coefficients.
Let $H$, $A$, and $P_n$ be as above.
There exists a homological spectral sequence 
\[
E^2_{p q} = H_p(K_q(G \ltimes P_\bullet)) \Rightarrow K_{p + q}(G \ltimes A).
\]
\end{theorem}

\begin{remark}\label{rem:spec-seq-gen-special-cases}
If we choose the subgroupoid $H$ to be the unit space $G^{(0)} \subset G$, the $E^1$-sheet of above spectral sequence becomes $E^1_{p q} = K_q(C_0(G^{(p)}) \otimes_{C_0(G^{(0)})} A)$.
If moreover we assume that $G$ is ample, the $E^2$-sheet can be identified with $H_p(G, K_q(A))$, the groupoid homology of $G$ with the coefficient $G$-sheaf corresponding to the $C_c(G, \Z)$-module $K_q(A)$~\cite{valmak:groupoid}*{Section 3}.
\end{remark}

\subsection{Smale spaces}
\label{sec:smale-sp}

A \emph{Smale space} is a compact metric space $X$ with a self-homeomorphism $\phi$ satisfying a certain hyperbolicity condition~\cite{ruelle:thermo}.
The most essential feature of Smale spaces is given by the existence of two equivalence relations on $X$, defined as follows: 
\begin{itemize} \item given $x,y\in X$, we say they are \emph{stably equivalent} (denoted $x\sim_s y$) if
\[
\lim_{n\to \infty} d(\phi^{n}(x),\phi^{n}(y)) = 0;
\]
\item given $x,y\in X$, we say they are \emph{unstably equivalent} (denoted $x\sim_u y$) if
\[
\lim_{n\to \infty} d(\phi^{-n}(x),\phi^{-n}(y)) = 0.
\] 
\end{itemize}

The graph of the unstable equivalence relation has a structure of locally compact groupoid with a Haar system~\cite{put:algSmale}, that we denote by $R^u(X, \phi)$.
Note that $R^s(X,\phi)=R^u(X,\phi^{-1})$ and $R^s(X,\phi)=R^s(X,\phi^k)$ for any $k\geq 1$.
The orbits of $R^s(X,\phi)$ are referred to as \emph{stable sets}.
Following the construction detailed in~\cite{put:spiel}, we obtain an étale groupoid by restricting $R^u(X,\phi)$ to an appropriate subspace contained in a finite union of stable sets.

\emph{We assume that a Smale space $(X,\phi)$ is non-wandering}, in the following sense.
\begin{definition}\label{def:nwirrmix}
A dynamical system $(X,\phi)$ is said to be 
\begin{itemize}
\item \emph{non-wandering} if for any $x\in X$ and every open set $U$ containing $x$, there is a positive integer $N$ such that $\phi^N(U) \cap U$ is non-empty;
\item \emph{irreducible} if, for every (ordered) pair of non-empty open sets $U, V$, there is a positive integer $N$ such that $\phi^N(U) \cap V$ is non-empty;
\item \emph{mixing} if for every (ordered) pair of non-empty open sets $U, V$, there is a positive integer $N$ such that
$\phi^n(U)\cap V$ is non-empty for all $n \geq N$.
\end{itemize}
\end{definition}

Given a finite directed graph $\cG = (\cG^0, \cG^1, i, t \colon \cG^1 \to \cG^0)$, the associated \emph{shift of finite type} is the Smale space with underlying metric space
\[
\Sigma_\cG = \{e = (e_k)_{k \in \Z} \in (\cG^1)^\Z \mid t(e_k) = i(e_{k + 1})\}
\]
and homeomorphism given by $\sigma(e)_k = e_{k+1}$.

A Smale space $(X, \phi)$ is isomorphic to a shift of finite type as above if and only if $X$ is totally disconnected, and $\phi$ is furthermore irreducible if and only if $\cG$ is connected (see~\cite{put:HoSmale}*{Theorem 2.2.8}).
When $(\Sigma, \sigma)$ is a shift of finite type, the groupoid $R^u(\Sigma, \sigma)$ is AF.
Moreover, the $K_0$-group of $C^* R^u(\Sigma, \sigma)$ is isomorphic to the \emph{stable dimension group}
$D^s(\Sigma, \sigma)$ introduced by Krieger~\cite{krieger:inv}, which can be described using stable sets of $(\Sigma, \sigma)$ (this is explained in detail in~\cite{put:HoSmale}*{Chapter 3}).
The $K_1$-group of $C^* R^u(\Sigma, \sigma)$ is trivial.

\subsection{Homology of Smale spaces}
\label{sec:homolog-Sm-sp}

The main reference for homology of Smale spaces is~\cite{put:HoSmale}.
Let $(X, \phi)$ be a non-wandering Smale space.
An \emph{$s/u$-bijective pair} over $(X, \phi)$ is given by a Smale space $(Y, \psi)$ with totally disconnected unstable sets and an $s$-bijective map $f \colon Y \to X$, and another Smale space $(Z, \zeta)$ with totally disconnected stable sets and a $u$-bijective map $g\colon Z \to X$.
Given such data, we set
\begin{align*}
Y_n &= \underbrace{Y \times_X \dots \times_X Y}_{(n+1) \times},&
\psi_n &= (\psi \times \dots \times \psi)|_{Y_n}.
\end{align*}
Then $(Y_n, \psi_n)$ is again a Smale space with totally disconnected unstable sets.
We denote the restriction of $f \times \dots \times f$ to $Y_n$ by $f_n$.
This map is again $s$-bijective~\cite{put:notes}*{Theorem 5.2.15}.
Similarly, we define a Smale space with totally disconnected stable sets,
\begin{align*}
Z_n &= \underbrace{Z \times_X \dots \times_X Z}_{(n+1) \times},&
\zeta_n &= (\zeta \times \dots \times \zeta)|_{Z_n},
\end{align*}
together with the $u$-bijective map $g_n\colon Z_n \to X$.
Then we have that
\begin{align*}
\Sigma_{L,M} &= Y_L \times_X Z_M, &
\sigma_{L,M} &= (\psi_L \times \zeta_M)|_{\Sigma_{L,M}}
\end{align*}
is a Smale space with totally disconnected underlying space.
Thus it is a shift of finite type.

By definition, the \emph{stable homology} of $(X, \phi)$, denoted by $H^s_\bullet(X, \phi)$, is the homology of the double complex with terms $D^s(\Sigma_{L,M}, \sigma_{L,M})$ for $L, M \ge 0$~\citelist{\cite{put:HoSmale}\cite{val:smale}}.
This construction is independent of the chosen $s/u$-bijective pair and it is functorial for $s$-bijective maps.

We will also work with a normalized double complex originally introduced in~\cite{put:HoSmale}*{Chapter 4}: the group $D^s(\Sigma_{L,M}, \sigma_{L,M})$ admits a natural action of $S_{L + 1} \times S_{M + 1}$.
We will use the notation $D^s(\Sigma_{L,M}, \sigma_{L,M})_{,A}$ to denote the subgroup of elements $x$ satisfying $s x = (-1)^{\absv{s}} x$ for $s \in S_{M + 1}$.
As a simple but useful remark, notice that when the map $g \colon Z \to X$ is at most $N$-to-$1$, the group $D^s(\Sigma_{L,M}, \sigma_{L,M})_{,A}$ is trivial for $M > N$.

Now we recall a useful result from~\cite{put:funct}*{Corollary 3.5}.
\begin{proposition}\label{prop:put35}
      Consider Smale spaces $(X,\phi)$ and $(Y,\psi)$, and an $s$-bijective map $g\colon (Y,\psi) \to (X,\phi)$.
      We then have étale groupoids $G=R^u(X,\phi)|_T$ and $H=R^u(Y,\psi)|_{T_0}$ for appropriate transversals $T \subset X$ and $T_0 \subset Y$ (see Assumption~\ref{ass:tr}), and such that $g(T_0) = T$ and $H$ embeds as an open subgroupoid of $G$.
\end{proposition}

The previous proposition can be applied for example when $X$ has totally disconnected stable sets and $(Y,\psi)=(\Sigma,\sigma)$ is the shift of finite type associated to a Markov partition of $X$~\cite{bs:dynbook}*{Section 5.12}.
Then the construction of Section~\ref{sec:triang-cat-from-groupoid-equiv} gives the following.

\begin{theorem}[\cite{valmak:groupoidtwo}*{Theorem 3.9}]\label{thm:spec-seq-Smale-tot-discon-st-set}
Let $(X,\phi)$ be a non-wandering Smale space with totally disconnected stable sets.
We then have a homological spectral sequence
\[
E^2_{p q} = E^3_{p q} = H^s_p(X, \phi) \otimes K_q(\bC) \Rightarrow K_{p + q}(C^*R^u(X,\phi)).
\]
\end{theorem}

In fact, Putnam's homology agrees with the groupoid homology:

\begin{theorem}[\cite{valmak:groupoidtwo}*{Theorem 4.1}]\label{thm:grpd-hmlg-compar-tot-disconn}
With $(X, \phi)$ and $G$ as above, we have 
\[ H^s_p(X, \phi) \cong H_p(G, \sfZ).\]
\end{theorem}

\section{Symmetric simplicial spaces}
\label{sec:symm-simp-sp}

Let us start with some definitions and conventions.
Let $A_\bullet = (A_n)_{n=0}^\infty$ be a \emph{simplicial topological space}, that is, a simplicial object in the category of topological spaces.
Its \emph{thin geometric realization} is defined to be
\[
\absv{A_\bullet} = \left( \coprod_{n = 0}^\infty \Delta^n \times A_n \right) / \sim,
\]
where $\Delta^n$ is the standard $n$-simplex, and $\sim$ is the equivalence relation induced by the structure maps (the degeneracy and face maps)~\cite{MR232393}*{Section 2}.

Next further assume that $A_\bullet$ is a \emph{symmetric simplicial space}, or equivalently, an $S_\bullet$-space for the \emph{crossed simplicial group} $S_\bullet$~\cite{MR998125}, i.e., a simplicial space with compatible actions of $S_{n+1}$ on $A_n$.
Alternatively, we can think of $A_\bullet$ as a contravariant functor on the category of (unordered) finite sets.
Then we can form a ``naive'' geometric realization $\nvGR{A_\bullet}$, as the quotient of $\absv{A_\bullet}$ with respect to the orbit relations for the diagonal action of $S_{n+1}$ on $\Delta^n \times A_n$, for all $n$.
Moreover, we denote the image of $\Delta^n \times A_n$ in $\nvGR{A_\bullet}$ by $\nvGR{A_\bullet}^{(n)}$.

\subsection{Adaptation of Segal's work}
\label{sec:adapt-segal}
In this section we leverage Segal's work~\cite{MR232393} on simplicial spaces and spectral sequences to obtain cohomological approximation results in the setting of groupoids and dynamical systems.

Let $T$ be a locally compact Hausdorff space, and $g\colon T_0 \to T$ be a proper continuous surjective map from a locally compact space $T_0$.
Then we obtain an $S_\bullet$-space $T_\bullet$ by setting $T_n$ to be the $(n+1)$-fold fiber product of $T_0$ over $T$.
Let us also denote by $T_n^f$ the subset of points $(x_0, \dots, x_n)$ such that $x_i \neq x_j$ for $i \neq j$, i.e.~the union of free $S_{n+1}$-orbits in $T_n$.
This is an open subset of $T_n$, and $\nvGR{T_\bullet}$ is (as a set) the disjoint union of the images of $\mathring{\Delta}^n \times T_n^f$ for different $n$.

In the rest of the section we assume that the map $g$ is at most $N$-to-one, for some $N$.
Under this assumption on $g$, we also have $\nvGR{T_\bullet} = \nvGR{T_\bullet}^{(N)}$.

\begin{proposition}
The space $\nvGR{T_\bullet}$ is a locally compact and Hausdorff topological space.
Moreover, the induced map $\tilde g \colon \nvGR{T_\bullet} \to T$ is proper.
\end{proposition}

\begin{proof}
Consider the space $Y$ of bounded positive linear functionals $\phi$ on the C$^*$-algebra $C_0(T_0)$, which are contractive, i.e., the ones satisfying $0 \le \phi(f)$ for nonnegative functions $f$ and $\phi(f) \le \left\| f \right\|$ for all $f$.
Under the weak topology for the natural pairing with $C_0(T_0)$, this is a compact Hausdorff space.

As a subspace of $Y$, we consider
\[
Y' = \{ \phi \in Y \mid \exists t \in T\; \forall f \in C_0(T) \colon \phi(g^* f) = f(t) \}.
\]
We are going to show that $Y'$ is a model of $\nvGR{T_\bullet}$, with the claimed properties.

Let us first check that $Y'$ is a locally closed subspace of $Y$, which would imply that $Y'$ is locally compact and Hausdorff.
Given a point $\phi \in Y'$, let us choose $f \in C_0(T)$ such that $\phi(g^* f) = 1$.
We define an open neighborhood $U$ of $\phi$ in $Y$ by
\[
U = \{ \psi \in Y \mid \absv{\psi(f) - 1} < 1 \},
\]
and a closed set $C \subset Y$ by
\[
C = \{ \psi \in Y \mid \forall f_1, f_2 \in C_0(T)\colon \psi(g^*(f_1 f_2)) = \psi(g^* f_1) \psi(g^* f_2) \}.
\]
Then we have $U \cap C = U \cap Y'$, which implies the claim.
Indeed, the defining condition of $C$ says that any $\psi \in U \cap C$ has an associated point $t_\psi \in T \cup \{ \infty \}$ such that $\psi(g^* f) = f(t_\psi)$ holds for all $f \in C_0(T)$, where we formally put $f(\infty) = 0$.
The defining condition for $U$ forbids $t_\psi = \infty$, hence we obtain the claim.

Let us fix $t \in T$, and consider
\[
Y'_t = \{ \phi \in Y' \mid \forall f \in C_0(T) \colon \phi(g^* f) = f(t) \}.
\]
This is a compact affine subspace of $Y$, and its extreme points are the evaluation functionals $\delta_{x}$ for $x$ in the fiber $g^{-1}(t) \subset T_0$.
By our assumption $g^{-1}(t)$ is a finite set, of cardinal at most $N$.
Thus, such a $\phi$ can be written as a convex combination
\[
\phi = \sum_{x \in g^{-1}(t)} a_{x} \delta_{x}.
\]

Now, let us write the standard simplex $\Delta^n$ as
\[
\Delta^n = \{ (a_0, \dots, a_n) \in \R^{n+1} \mid \sum_i a_n = 1, a_i \ge 0 \}.
\]
Then the map
\[
T_n \times \Delta^n \to Y', \quad ((x_0, \dots, x_n), (a_0, \dots, a_n)) \mapsto \sum_{i = 0}^n a_i \delta_{x_i}
\]
is continuous, and for different $n$, these maps are compatible with the gluing conditions for the naive realization.
This shows that there is a continuous map $h \colon \nvGR{T_\bullet} \to Y'$.

By a routine argument, we see that the map $h$ is bijective, hence $\nvGR{T_\bullet}$ is a Hausdorff topological space.
It remains to show that $h$ is a homeomorphism, and that $\tilde{g}$ is proper.

Let us fix $t \in T$, and its compact neighborhood $K \subset T$.
Then $K_0 = g^{-1}(K)$ is a compact subset of $T_0$ by the properness of $g$, and the same holds for $K_n = g_n^{-1}(K) \subset T_n$ for the induced map $g_n \colon T_n \to T$.
Then $\tilde{g}^{-1}(K)$ is a compact neighborhood of $\tilde{g}^{-1}(t)$ in $\nvGR{T_\bullet}$, being a quotient of the disjoint union $K_0 \coprod \dots \coprod K_N$, which is a compact space.
Since any continuous bijection from a compact space to a Hausdorff space is a homeomorphism, we obtain that $h$ restricts to a homeomorphism from $\tilde{g}^{-1}(K)$ to its image in $Y'$.
This shows that $h$ is indeed a homeomorphism.
Moreover, since we checked that $\tilde{g}^{-1}(K)$ is compact for compact $K$, $\tilde{g}$ is proper.
\end{proof}

\begin{proposition}\label{prop:sheaf-coh-emb-to-skel}
We have an isomorphism
\[
H^n(T, \sfZ) \cong H^n(\nvGR{T_\bullet}, \sfZ)
\]
induced by $\tilde g$.
\end{proposition}

\begin{proof}
For a topological space $X$, let $D^+(X)$ be the derived category of cochain complexes of sheaves of commutative groups which are bounded below up to quasi-isomorphisms.
We thus have the right derived functor $\cR \tilde g_*\colon D^+(\nvGR{T_\bullet}) \to D^+(T)$ of the direct image functor $\tilde g_*$.
We also have the (term-wise application of) inverse image functor $\tilde g^*\colon D^+(T) \to D^+(\nvGR{T_\bullet})$.

Since $\sfZ_{\nvGR{T_\bullet}} =\tilde g^* \sfZ_T$, we have a natural morphism $\sfZ_T \to \cR \tilde g_* \sfZ_{\nvGR{T_\bullet}}$ in $D^+(T)$ induced by the adjunction between $\cR \tilde g_*$ and $\tilde g^*$.

We first claim that $R^n \tilde g_* \sfZ_{\nvGR{T_\bullet}} = 0$ for $0 < n$.
By the properness of $\tilde g$, for any $x \in T$ we have
\[
(R^n \tilde g_* \sfZ_{\nvGR{T_\bullet}})_x = H^n(\tilde g^{-1}(x), \sfZ)
\]
from the proper base change theorem, see~\cite{sga4-vbis}*{Corollaire 4.1.2}.
Here $\tilde g^{-1}(x)$ is homeomorphic to $\Delta^m$ with $m = \absv{g^{-1}(x)}$, hence the $n$-th cohomology group on the right hand side must vanish.
This also shows $\tilde g_* \sfZ_{\nvGR{T_\bullet}} = \sfZ_T$.

Now, we have the Leray spectral sequence
\[
E_2^{p, q} = H^p(T, R^q \tilde g_* \sfZ_{\nvGR{T_\bullet}}) \Rightarrow H^{p+q}(\nvGR{T_\bullet}, \sfZ_{\nvGR{T_\bullet}}).
\]
By the above claim we have $E_2^{p, q} = 0$ for $0 < q$, and $E_2^{p, 0} = H^p(T, \sfZ_T)$.
As the higher differentials are $d_r^{p, q}\colon E_r^{p, q} \to E_r^{p+r,q-r+1}$, we have $E_2^{p, 0} = E_\infty^{p, 0}$ and $E_\infty^{p,q} = 0$ for $q > 0$.
This implies
\[
H^n(\nvGR{T_\bullet}, \sfZ_{\nvGR{T_\bullet}}) \cong E^{n,0}_\infty \cong H^n(T, \sfZ_T),
\]
as claimed.
\end{proof}

\begin{proposition}\label{prop:naive-geom-realiz-K-group-isom}
We have $K^\bullet(\nvGR{T_\bullet}) \cong K^\bullet(T)$ induced by $\tilde g$.
\end{proposition}

\begin{proof}
First, we may assume that $T$ is compact.
Indeed, denoting the one-point compactification by $T^+$, we have a proper surjective map $T_0^+ \to T^+$.
Then we have $\nvGR{T_\bullet}^+ = \nvGR{T_\bullet^+}$, hence it is enough to prove the claim for $T_\bullet^+$.

Consider the representable $K$-groups $\RK^n(X) = [X, \KU_n]$, where $\KU_n$ is the $K$-theory spectrum, i.e., $\KU_{2k} = BU \times \Z$ and $\KU_{2k+1} = \Omega BU = U$ for the infinite unitary group $U = \varinjlim U_k$.
Then, this is a generalized cohomology as considered in~\cite{MR232393}, while satisfying $\RK^n(X) \cong K^n(X)$ for compact $X$.
We thus need to check $\RK^n(\nvGR{T_\bullet}) \cong \RK^n(T)$.

By~\cite{MR232393}*{Proposition 5.2}, there is a spectral sequence
\[
E_2^{p q} = H^p(X, F^q) \Rightarrow \RK^{p + q}(X).
\]
where $F^{2 k} = \sfZ$ and $F^{2 k + 1} = 0$.
(This is a version of the Atiyah--Hirzebruch spectral sequence that works without any CW-complex structure.)
By Proposition~\ref{prop:sheaf-coh-emb-to-skel}, we have the isomorphism of $E_2$-sheet for $X = T$ and $X = \nvGR{T_\bullet}$.
\end{proof}

Following the scheme in~\cite{MR232393}, we define a spectral sequence converging to $K_\bullet(G \ltimes C_0(\nvGR{T_\bullet}))$.
The filtration $\nvGR{T_\bullet}^{(n)}$ on $\nvGR{T_\bullet}$, combined with Proposition~\ref{prop:naive-geom-realiz-K-group-isom} induces a spectral sequence.
\[
E_1^{p q} = K^{p+q}(\nvGR{T_\bullet}^{(p)} \setminus \nvGR{T_\bullet}^{(p-1)}) \Rightarrow K^{p + q}(T).
\]

When $V$ is a vector space (over $\Q$) with a linear representation of $S_n$, let us denote its isotypic summand for the sign representation by $V_A$.
This is the image of the projection
\[
p_a = \frac1{n!} \sum_{s \in S_n} (-1)^{\absv{s}}s
\]
acting on $V$.

\begin{lemma}\label{lem:triv-alt-for-non-free-action}
Let $X$ be a locally compact Hausdorff space with an action of $S_n$, such that the stabilizer of any point $x \in X$ contains an odd permutation.
We then have $H^\bullet(X, \underline{\Q})_A = 0$.
\end{lemma}

\begin{proof}
Let $Y = X / S_n$, and consider the projection map $\pi \colon X \to Y$.
In this situation, we claim that $H^\bullet(X, \underline{\Q}) \cong H^\bullet(Y, \pi_* \underline{\Q})$.
Indeed, $\pi$ is a proper map with discrete fiber, hence the proper base change theorem applies and gives $\cR \pi_* \underline{\Q} \simeq \pi_* \underline{\Q}$.
Then the claim is a straightforward consequence of $\cR \pi^Y_* \cR \pi_* \simeq \cR \pi^X_*$ for the canonical maps $\pi^X$ and $\pi^Y$ from $X$ and $Y$ to the singleton space.

Acting by the projection $p_a$, we obtain
\[
H^\bullet(X, \underline{\Q})_A \cong H^\bullet(Y, p_a \pi_* \underline{\Q})
\]
Now, at the level of stalks we have
\[
(p_a \pi_* \underline{\Q})_y = \Gamma(\pi^{-1}(y), \Q)_A
\]
for all $y \in Y$.
This has to be trivial by our assumption on stabilizers.
Indeed, if $x$ is in $\pi^{-1}(y)$, we can choose an odd permutation $s$ from a stabilizer of $x$.
If $\sigma$ is in $\Gamma(\pi^{-1}(y), \Q)_A$, we have $s \sigma = - \sigma$, and evaluating at $x$ gives $\sigma(x) = 0$.
This implies $H^\bullet(Y, p_a \pi_* \underline{\Q}) = 0$, hence the claim.
\end{proof}

\begin{proposition}\label{prop:rat-segal-e1-sheet}
We have
\[
(K^{p+q}(\nvGR{T_\bullet}^{(p)} \setminus \nvGR{T_\bullet}^{(p-1)})) \otimes \Q \cong (K^{q}(T_p) \otimes \Q)_A
\]
with respect to the natural action of $S_{p+1}$ on $T_p$.
\end{proposition}

\begin{proof}
By construction, $\nvGR{T_\bullet}^{(p)} \setminus \nvGR{T_\bullet}^{(p-1)}$ is the quotient of $\mathring{\Delta}^p \times T_p^f$ by the diagonal action of $S_{p+1}$.
By~\cite{MR2820377}*{Proposition 5.5}, we have
\[
(K^{p+q}(\nvGR{T_\bullet}^{(p)} \setminus \nvGR{T_\bullet}^{(p-1)})) \otimes \Q \cong (K^{p+q}(\mathring{\Delta}^p \times T_p^f)^{S_{p+1}}) \otimes \Q.
\]

On the right hand side, $\mathring{\Delta}^p$ is homeomorphic to $\R^p$, and the Bott periodicity gives an isomorphism $K^{p+q}(\mathring{\Delta}^p \times T_p^f) \simeq K^q(T_p^f)$.
As the action of odd permutations on $\mathring{\Delta}^p$ flips the orientation, the action of $S_{p+1}$ on $K^{p+q}(\mathring{\Delta}^p \times T_p^f)$ can be identified with the twist of the natural one on $K^q(T_p^f)$ by the sign representation under this isomorphism.
Moreover, rationalization and taking $S_{p+1}$-invariants commute, hence we have $(K^q(T_p^f) \otimes \Q)_A$.

It remains to show $(K^q(T_p^f) \otimes \Q)_A \cong (K^q(T_p) \otimes \Q)_A$.
By the $6$-term exact sequence involving $K^\bullet(T_p^f) \otimes \Q$, $K^\bullet(T_p) \otimes \Q$, and $K^\bullet(T_p \setminus T_p^f) \otimes \Q$, which is $S_{p+1}$-equivariant, the claim is equivalent to $(K^\bullet(T_p \setminus T_p^f) \otimes \Q)_A = 0$.

Let $X$ be the one point compactification of $T_p \setminus T_p^f$.
Then Lemma~\ref{lem:triv-alt-for-non-free-action} implies that
\[
(K^i(X) \otimes \Q)_A \cong \biggl(\bigoplus_{k} H^{i + 2 k}(X, \underline{\Q})\biggr)_A
\]
is trivial for $i = 0, 1$, hence we obtain the claim.
\end{proof}

\begin{corollary}
We have a convergent cohomological spectral sequence
\[
E_1^{p q} = (K^q(T_p) \otimes \Q)_A \Rightarrow K^{p + q}(T) \otimes \Q,
\]
with the $E_1$-differential $E_1^{p q} \to E_1^{(p+1) q}$ induced by the simplicial structure of $T_\bullet$.
\end{corollary}

\subsection{Equivariant \texorpdfstring{$K$}{K}-theory spectral sequence}
\label{sec:equiv-k-th-spec-seq}

We now consider the case where $G$ is an étale groupoid with torsion-free stabilizers satisfying the strong Baum--Connes conjecture~\cite{val:kthpgrp}, with base $T = G^{(0)}$.
Suppose that $T_0$ is a $G$-space, mapping onto $T$ by a surjective map which is at most $N$-to-one.
Then $\nvGR{T_\bullet}$ is also a locally compact $G$-space, and we have an equivariant $*$-homomorphism $C_0(T) \to C_0(\nvGR{T_\bullet})$.

\begin{proposition}\label{prop:k-grp-isom-for-naive-geom-real-cross-prod}
Under the above setting, the inclusion $C^*_r G \to G \ltimes C_0(\nvGR{T_\bullet})$ induces an isomorphism $K_\bullet(G \ltimes C_0(\nvGR{T_\bullet})) \cong K_\bullet(C^*_r G)$.
\end{proposition}

\begin{proof}
By Theorem~\ref{thm:spec-seq-gen} and Remark~\ref{rem:spec-seq-gen-special-cases}, we have a spectral sequence
\[
E^1_{p q} = K^q(G^{(p)} \times_T \nvGR{T_\bullet}) \cong K_q(G^{(p)} \ltimes C_0(\nvGR{T_\bullet})) \Rightarrow K_{p + q}(G \ltimes C_0(\nvGR{T_\bullet})),
\]
and a similar one for $C^*_r G$.
Now, we observe that $G^{(p)} \times_T T_0$ maps onto $G^{(p)}$ by a finite-to-one map, and
\[
G^{(p)} \times_T \nvGR{T_\bullet} \cong \nvGR{G^{(p)} \times_T T_\bullet}.
\]
Combining this with Proposition~\ref{prop:naive-geom-realiz-K-group-isom}, we have
\[
K^q(G^{(p)} \times_T \nvGR{T_\bullet}) \cong K^q(G^{(p)}).
\]
Thus, the inclusion $C^*_r G \to G \ltimes C_0(\nvGR{T_\bullet})$ induces an isomorphism of spectral sequences at the $E^1$-sheet.
\end{proof}

This time we obtain the convergent spectral sequence
\begin{equation}\label{eq:spec-seq-for-K-of-G-crossed}
E_1^{p q} = K_{p+q}(G \ltimes C_0(\nvGR{T_\bullet}^{(p)} \smallsetminus \nvGR{T_\bullet}^{(p-1)})) \Rightarrow K_{p + q}(C^*_r G).
\end{equation}

\subsection{Soft resolution from simplicial totally disconnected space}
\label{sec:sof-res-from-u-bij-map}

Now, suppose $T_0$ is a totally disconnected space with a surjective proper continuous map $g\colon T_0 \to T$.
We denote the $(n+1)$-fold fiber product of $T_0$ over $T$ by $T_n$, and by $g_n$ the induced map $T_n \to T$.

\begin{proposition}\label{prop:dir-img-soft}
Let $F$ be a sheaf over $T_n$.
Then $(g_n)_* F$ is a (c-)soft sheaf on $T$.
\end{proposition}

\begin{proof}
To check the c-softness of $(g_n)_* F$, we want to show the following property (see~\cite{MR0345092}*{Théorème II.3.4.1}): for any point $x \in T$, there is a neighborhood $U$ of $x$ such that, for any closed set $S$ of $T$ contained in $U$ and $s \in \Gamma(S, (g_n)_* F)$, there is an extension $\tilde s \in \Gamma(U, (g_n)_* F)$ of $s$.
We take $U$ to be a relatively compact open neighborhood of $x$.

By the second countability and total disconnectedness of $T_n$, $F$ is soft.
Then, if $S$ is a closed set as above (hence compact), any section of $F$ on the compact set $g_n^{-1}(S)$ extends to a section on $g_n^{-1}(U)$.
We thus need to show that the sections of $(g_n)_* F$ on $S$ can be identified with those of $F$ on $g_n^{-1}(S)$.

We have (\cite{MR0345092}*{II.3.3.1})
\begin{align*}
\Gamma(S, (g_n)_* F) &= \varinjlim_{V \supset S} \Gamma(g_n^{-1}(V), F),&
\Gamma(g_n^{-1}(S), F) &= \varinjlim_{V' \supset g_n^{-1}(S)} \Gamma(V', F),
\end{align*}
where $V$ runs over open subsets of $U$ containing $S$, while $V'$ runs over those containing $g_n^{-1}(S)$.
We thus have the equality of these spaces if the sets $g_n^{-1}(V)$ form a fundamental system of neighborhoods of $g_n^{-1}(S)$.

Let $V'$ be an open subset of $g_n^{-1}(U)$ containing $g_n^{-1}(S)$.
Then $W = g_n^{-1}(\bar U) \smallsetminus V'$ is a compact subset of $T_n$ by the properness of $g_n$.
Then $V = U \smallsetminus g_n(W)$ is an open neighborhood of $S$ such that $g_n^{-1}(V) \subset V'$.
\end{proof}

We also have the following correspondence for compactly supported sections.

\begin{proposition}\label{prop:cpt-section-dir-img}
Let $F$ be a sheaf of commutative groups on $T_n$.
Then $\Gamma_c(S, (g_n)_* F) = \Gamma_c(g_n^{-1}(S), F)$ for any open set $S \subset T$.
\end{proposition}

\begin{proof}
This follows from the properness of $g_n$.
\end{proof}

Now, let $F$ be a sheaf of commutative groups over $T$, and put $F^n = (g_n)_*(g_n^* F)$.
This is a soft sheaf on $T$ by Proposition~\ref{prop:dir-img-soft}.
Moreover, pullback along the structure maps of $(T_n)_n$ as a simplicial space induces a structure of cosimplicial object on $(F^n)_{n=0}^\infty$ in the category of sheaves over $T$.
Namely, let $d^n_i\colon T_n \to T_{n-1}$ and $s^n_i\colon T_n \to T_{n+1}$ be the face and degeneracy maps of $T_\bullet$.
Then for any open set $U \subset T$,
\[
\delta^n_i|_U\colon \Gamma(g_{n-1}^{-1}(U), g_{n-1}^* F) = \Gamma(U, F^{n-1}) \to \Gamma(g_{n}^{-1}(U), g_n^* F) = \Gamma(U, F^n), \quad \sigma \mapsto \sigma \circ d^n_i,
\]
defines a morphism $\delta^n_i \colon F^{n-1} \to F^n$.
Similarly we get $\sigma^n_i \colon F^{n+1} \to F^n$ from $s_i^n$, and we obtain a cosimplicial structure on $(F^n)_{n=0}^\infty$.

\begin{proposition}\label{prop:cosimp-res}
The cochain complex of $(F^n)_{n=0}^\infty$ is a resolution of $F$.
\end{proposition}

\begin{proof}
We need to check that, for each $t \in T$, the sequence of stalks
\[
0 \to F_t \to F^0_t \to F^1_t \to \dots
\]
is exact.
Put $Z_t = g_0^{-1}(t)$, which is a nonempty finite set by our assumption on $g_0 = g$.
The structure of $T_\bullet$ restricts to the standard simplicial structure on the direct products $(Z_t^{n+1})_{n=0}^\infty$, which in turn gives the above complex via the identification $F_t^n = C(Z_t^{n + 1}, F_t)$.
This simplicial set is contractible: for example, by fixing a basepoint $z \in Z_t$, we obtain an `extra degeneracy' map
\[
s_{-1} \colon Z_t^{n + 1} \to Z_t^{n + 2}, \quad (z_0, \dots, z_n) \mapsto (z, z_0, \ldots, z_n),
\]
in the sense of~\cite{MR2840650}*{Section III.5}, which implements a homotopy to the singleton.
From this we obtain the exactness of the above complex.
\end{proof}

Now, we further assume that $g$ is at most $N$-to-one.
Let us describe a soft resolution of bounded length, given by the ``alternating'' variation of the above construction.
The natural action of $S_{n+1}$ on $T_n$ induces an action on $F^n$.
We define a subsheaf $F^{n,a} \subset F^n$ as follows:
\[
\Gamma(U, F^{n,a}) = \{ x \in \Gamma(U, F^n) \mid \supp x \subset T_n^f, \forall s \in S_{n+1}\colon s x = (-1)^{\absv{s}} x \}.
\]

\begin{proposition}
The sheaf $F^{n,a}$ is again c-soft.
\end{proposition}

\begin{proof}
As in the proof of Proposition~\ref{prop:dir-img-soft}, let $U$ be a relatively compact open set of $T$, and $S$ be a compact subset of $U$.
Given a section $x \in \Gamma(S, F^{n,a})$, we need to find an extension $\tilde x \in \Gamma(V, F^{n,a})$ of $x$ to some open neighborhood $V$ of $S$.

By Proposition~\ref{prop:dir-img-soft}, we can find an open neighborhood $V_0$ of $S$ and a section $\tilde x_0 \in \Gamma(V_0, F^n)$ restricting to $x$.
Let $\bar x_0 \in \Gamma(g_n^{-1}(V_0), g_n^* F)$ be the section corresponding to $\tilde x_0$.
It is enough to find a $S_{n+1}$-invariant open neighborhood $V'$ of $g_n^{-1}(S)$ in $g_n^{-1}(V_0)$ such that the restriction $\bar x_0|_{V'}$ satisfies $s(\bar x_0|_{V'}) = (-1)^{\absv{s}} \bar x_0|_{V'}$ for $s \in S_{n+1}$.
Indeed, the argument of the proof of Proposition~\ref{prop:dir-img-soft} gives a neighborhood $V$ of $S$ satisfying $g_n^{-1}(V) \subset V'$, and the restriction of $\tilde x_0$ to $V$ is a section of $F^{n,a}$.

For each $u \in g_n^{-1}(S)$, find an open neighborhood $V_u$ such that $s^* \tilde x_0 = (-1)^{\absv{s}} \tilde x_0$ holds on $V_u$ for any $s \in S_{n+1}$.
Moreover, if $u$ is not in $T_n^f$, we ask $\tilde x_0$ to be trivial on $V_u$.
This is possible since the stalks $s^* \tilde x_0(u)$ are the inductive limits of $s^* \tilde x_0$ around open neighborhoods of $u$.
By compactness there are finitely many $u_1, \dots, u_k$ such that $V' = \bigcup_i V_{u_i}$ covers $g_n^{-1}(S)$.
\end{proof}

A standard argument shows that $(F^{n,a})_n$ is a subcomplex of $(F^n)_n$.
By assumption about the support and $T_n^f$, this is a bounded complex.

\begin{proposition}
The complex $(F^{n,a})_n$ is still a resolution of $F$.
\end{proposition}

\begin{proof}
Again we need to check the exactness of the augmented complex of stalks,
\[
0 \to F_t \to F^{0,a}_t \to F^{1,a}_t \to \dots \to F^{N,a}_t \to 0.
\]
The stalk of $F^{n,a}$ at $t$ is given by
\[
F^{n,a}_t = \{f \in C(Z_t^{n+1}, F_t) \mid \supp f \subset Z_t^{n+1} \smallsetminus D_{n+1}, \forall s \in S_{n+1} \colon s f = (-1)^{\absv{s}} f \}
\]
for $Z_t = g_0^{-1}(t)$, where $D_{n+1}$ is the degenerate part (the union of non-free orbits).
A usual normalization argument for simplicial complexes (for example~\cite{eilstee:algtop}*{Theorem VI.6.9}), together with the fact that $(Z_t^{n+1})_{n=0}^\infty$ is a contractible simplicial set, shows that the above augmented complex is indeed exact.
\end{proof}

Now, suppose that $G$ is an étale groupoid, and $T_0$ is a totally disconnected $G$-space with a proper structure map $g\colon T_0 \to T$ which is at most $N$-to-one.
If $F$ is a $G$-sheaf of commutative groups on $T$, the above construction $F^{\bullet, a}$ gives a c-soft resolution of $G$-sheaves.
Thus, the groupoid homology for $G$-sheaves are definable.
Moreover, we have
\begin{equation}\label{eq:cpt-sec-on-G-k}
\Gamma_c((G \ltimes T_n)^{(k)}, s^* g_n^* F) \cong \Gamma_c(G^{(k)}, s^* F^n)
\end{equation}
by Proposition~\ref{prop:cpt-section-dir-img}.
This allows us to calculate $H_\bullet(G, F)$, as follows.

\begin{proposition}\label{prop:grpd-homology-from-ubij-map}
Let $F$ be a $G$-sheaf of commutative groups.
Then the groupoid homology $H_\bullet(G, F)$ is well-defined as the homology of the double complex with components 
\[
\Gamma_c(G^{(k)}, s^* F^{n,a}),\qquad n,k \ge 0
\]
with total degree $k - n$.
\end{proposition}

\section{Comparison of homologies}
\label{sec:comp-hom}

Now we are ready to apply the above machinery to Smale spaces.
Let us start with some preliminaries.
With definitions as in Section~\ref{sec:homolog-Sm-sp}, we fix an $s/u$-bijective pair $(Y, \psi, f, Z, \zeta, g)$ over a Smale space $(X, \phi)$.
Consider a subspace $T \subseteq X$ and suppose it is a transversal for the unstable equivalence relation.
By Proposition~\ref{prop:grpd-homology-from-ubij-map}, the étale groupoid $G = R^u(X, \phi)|_T$ indeed satisfies the assumption about $G$-sheaves to define homology.

\begin{proposition}\label{prop:lift-transv}
We can choose $T$ so that $T_n = g_n^{-1}(T)$ is a transversal for the unstable equivalence relation on $Z_n$ for all $n$.
\end{proposition}

\begin{proof}
Suppose first that $(X,\phi)$ is mixing (in particular any stable or unstable orbit is dense). We choose a periodic point $x_0 \in X$, and $T$ to be a local stable set around $x_0$.
Then $T_n$ is a union of local stable sets around the points of $g_n^{-1}(x_0)$~\cite{put:notes}*{Lemma 5.2.10}.
Now, let $z \in Z_n$.
Since $T$ is a transversal, we can find $x_1 \in T$ such that $g_n(z) \sim_u x_1$.
As $g_n$ is $u$-bijective, we can find $z_1 \sim_u z$ such that $g_n(z_1) = x_1$.

If the system is not mixing, we use Smale's decomposition theorem, which can be found in~\cite{put:funct}*{Theorem 2.1 and the following discussion}, and apply the argument above to each mixing factor resulting from the decomposition.
\end{proof}

In the following we will always use $T$ satisfying the claim of Proposition~\ref{prop:lift-transv}.
Put $G = R^u(X, \phi)|_T$.

\begin{proposition}
The transversal $T_n \subset Z_n$ has an action of $G$ with anchor map $g_n\colon T_n \to T$.
We have $G \ltimes T_n \cong R^u(Z_n, \zeta_n)|_{T_n}$.
\end{proposition}

\begin{proof}
Let $z \in T_n$, $x = g_n(z)$, and $(x', x) \in G$.
Since $g_n$ is $u$-bijective, there is a unique point $z' \in Z_n$ such that $z' \sim_u z$ and $g(z') = x'$.
We define the action of $(x', x)$ on $z$ to be this $z'$.
Conversely, if $z, z' \in T_n$ are unstably equivalent, $z'$ is the image of the action of $(g(z'), g(z))$ on $z$.
This gives the identification of groupoids.
\end{proof}

Recall that $g_n$ is proper on stable sets~\cite{put:notes}*{Theorem 5.2.5}.
In particular, $g_n$ is proper as a map $T_n \to T$.

Let us establish a generalization of Theorem~\ref{thm:grpd-hmlg-compar-tot-disconn}.

\begin{lemma}\label{lem:Sigma-L-M-is-fib-prod-of-Sigma-0-M}
For each $M$, the natural map $\Sigma_{0,M} \to Z_M$ is $s$-bijective.
Moreover, we have
\[
\Sigma_{L, M} = \underbrace{\Sigma_{0,M} \times_{Z_M} \dots \times_{Z_M} \Sigma_{0,M}}_{(L+1) \times}.
\]
\end{lemma}

\begin{proof}
Let $z = (z_0, \dots, z_M)$ and $z' = (z'_0, \dots, z'_M)$ be points of $Z_M$ which are stably equivalent.
This condition is equivalent to component-wise equivalences
\[
z_0 \sim_s z'_0, \quad \dots, \quad z_M \sim_s z'_M.
\]

Suppose that we are given a point in the fiber $(\Sigma_{0,M})_{z}$, that is, a tuple of the form
\[
(y_0, z_0, \dots, z_M) \quad (f(y_0) = g(z_0) = \dots = g(z_M)).
\]
Write $x = f(y_0) = g(z_i)$ and $x' = g(z'_i)$.
Then there is a unique point $y'_0$ in the stable equivalence class of $y$ which maps to $x'$.
Then $(y'_0, z'_0, \dots, z'_M)$ is the unique point of $\Sigma_{L,M}$ stably equivalent to $(y_0, z_0, \dots, z_M)$ and maps to $z'$.
The relation between $\Sigma_{L, M}$ and the fibered product of $\Sigma_{0,M}$ over $Z_M$ is obvious.
\end{proof}

Now, take a transversal $T \subset X$ as in Proposition~\ref{prop:lift-transv}, and consider the étale groupoid $G = R^u(X, \phi)|_T$.
On the one hand, $T_i  = g_i^{-1}(T) \subset Z_i$ is a totally disconnected $G$-space.
On the other, by virtue of Proposition~\ref{prop:put35} the factor map $Y_j \to X$ defines an open subgroupoid $H_j < G$ for each $j$. Similarly, $\Sigma_{j, i} \to Z_i$ defines an open subgroupoid $H_{j, i} < 
G \ltimes T_i$.

\begin{proposition}
The groupoid $H_{j}$ is Morita equivalent to $H_0^{\times_G (j+1)}$.
Moreover, $H_{j, i}$ is the transformation groupoid $H_j \ltimes T_i$, which is also Morita equivalent to the multiple groupoid pullback of $(j+1)$ copies of $H_{0, i}$ over $G \ltimes T_i$.
\end{proposition}

\begin{proof}
The main ingredient of the proof is the following transversality~\cite{valmak:groupoidtwo}*{Proposition 2.9}: given $y_0, \dots y_j$ such that $f(y_i)$ are mutually unstably equivalent in $X$, we can find $y'_0, \dots, y'_j$ such that $y'_i$ is unstably equivalent to $y_i$ in $Y$, and additionally $f(y'_i) = f(y'_0)$, for each $i$.
The first statement is proved in~\cite{valmak:groupoidtwo}*{Theorem 2.8}.
The rest can be proved in similar ways.
\end{proof}

As before, let us denote by $\sfZ$ the constant sheaf of integers over $X$, which is $G$-equivariant.

\begin{theorem}\label{thm:compar-homologies}
We have $H_k(G, \sfZ) \cong H^s_k(X, \phi)$ for all $k \in \Z$.
\end{theorem}

\begin{proof}
Consider the inclusion map $f_j\colon H^{\times_G (j+1)} \to G^{\times_G (j+1)}$.
We obtain a bicomplex of $G^{\times_G (j+1)}$-sheaves $F_{j \bullet}^\bullet = \cL (f_j)_! \sfZ$~\cite{cramo:hom}*{Section 4.2}.
This is defined through a choice of c-soft resolution $S^\bullet$ of $\sfZ$ as $H^{\times_G (j+1)}$-sheaves, then by taking the bicomplex $B_\bullet(f_j, S^\bullet)$, where $B_n(f_j, F)$ is defined as $(\beta_n)_! \alpha_n^* F$ for
\[
\begin{tikzcd}
G^{\prime(0)} & K^{(n)} \times_{G^{\prime(0)}} G' \arrow[r, "\alpha_n = r \pi_1"] \arrow[l, "\beta_n = s \pi_2"'] & K^{(0)}
\end{tikzcd}
\]
with $K = H^{\times_G (j+1)}$, $G' = G^{\times_G (j+1)}$, and $\pi_i$ denoting the projections to $K^{(n)}$ and $G'$.

By the Morita equivalence of Proposition~\ref{prop:triv-mult-pllback}, we can regard $F_{j \bullet}^\bullet$ as a complex of $G$-sheaves.
Suppose that $S^i$ is given by a bounded c-soft resolution of $\sfZ$ as a $G^{\times_G (j+1)}$-sheaf.
Then the stalk of $F_{j n}^i$ at $x \in G^{(0)}$ is given by
\begin{equation}\label{eq:F-j-i-x}
(F_{j n}^i)_x = F_{j n}(S^i)_x = \Gamma_c((H^{\times_G (j+1)})^{(n)} \times_{G^{(j)}} (G^{\times_G (j+1)})_x, S^i_x),
\end{equation}
where we identified the base of $G^{\times_G (j+1)}$ with $G^{(j)}$, and $x$ with $(x, \id_x, \dots, x)$ to make sense of $(G^{\times_G (j+1)})_x$.

Recall that we have a bounded c-soft resolution $F^{\bullet,a}$ of the $G$-sheaf $\sfZ$ from the simplicial totally disconnected $G$-space $T_\bullet$.
We will use the corresponding $G^{\times_G (j+1)}$-sheaves $S^\bullet$ in the above scheme. Moreover, we get a homological differential in the $j$-direction from the simplicial structure, hence we obtain a triple complex $F_{\bullet \bullet}^\bullet$ of c-soft $G$-sheaves.
We compute the groupoid hyperhomology $\bH_\bullet(G, F_{\bullet \bullet}^\bullet)$ in two ways to obtain the claim of the theorem.

To relate this to the groupoid homology $H_\bullet(G, \sfZ)$, it is enough to see that the complex $F_\bullet^\bullet$ is quasi-isomorphic to $\sfZ$ (at degree $0$), or equivalently, to the complex $F^{\bullet,a}$.

Given a $G$-sheaf $S$ and any $n$, we claim that the complex of $G$-sheaves $F_{\bullet n}(S)$, defined analogously to $F_{j n}^i$ above but with $S$ instead of $S^i$, form a homological resolution of $S$.
This is a stalk-wise computation, and we just need to check that the complex
\[
\cdots \to F_{1 n}(S)_x \to F_{0 n}(S)_x \to S_x \to 0
\]
is exact.
From \eqref{eq:F-j-i-x}, we see that $F_j(S)_x$ is the space of finitely supported maps from $A_x^{j+1}$ to $S_x$, where
\[
A_x = \{ (h, g) \mid h \in H^{(n)}, g \in G_x, s(h) = r(g) \}.
\]
Then we can construct a contracting homotopy $h\colon F_{j n}(S)_x \to F_{j+1, n}(S)_x$ for the above complex, by fixing $a \in A_x$ and setting
\[
h(c)(a_0, \dots, a_{j+1}) = \delta_{a, a_0} c(a_1, \dots, a_{j+1}).
\]

Next, if we treat $F_{\bullet \bullet}(S)$ as a bicomplex, this is still a resolution of $S$.
Indeed the above already shows that this is quasi-isomorphic to a complex with terms $S$ at each nonnegative degree.
Taking into account the differential in the $n$-direction, we see that it is
\[
\begin{tikzcd}
 \dots \arrow[r,"0"] & S \arrow[r, "\id"] & S \arrow[r, "0"] & S \arrow[r] & 0.
\end{tikzcd}
\]
Using the above argument for $S = F^{i,a}$, we get that $F^\bullet_{\bullet \bullet} = (F^i_{j n})_{i, j, n}$ is a resolution of $F^{\bullet,a}$, hence we get the isomorphism $\bH_\bullet(G, F_{\bullet \bullet}^\bullet) \cong H_\bullet(G, \sfZ)$.

It remains to show that our hyperhomology also computes $H^s_\bullet(X, \phi)$.
By definition $\bH(G, F^\bullet_{\bullet \bullet})$ is the homology of the quadruple complex with terms
\[
E_{j n k}^i = \Gamma_c(G^{(k)}, s^* F_{j n}^i)
\]
with total degree $j + n + k - i$.
The claim follows if we can show that the homology in the $(n,k)$-direction gives $D(\Sigma_{j,i})_{,A}$ concentrated at degree $0$ when we fix $j$ and $i$.

By the Morita equivalence between $H_j$ and $H^{\times_G (j+1)}$, this homology can be interpreted as $\bH_\bullet(G, \cL(f_j)_! F^{i,a})$ if we regard $F^{i,a}$ as an $H_j$-sheaf.
Then by $\bH_\bullet(G, \cL(f_j)_! F^{i,a}) \cong H_\bullet(H_j, F^{i,a})$, our claim reduces to the computation of $H_\bullet(H_j, F^{i,a})$.

Let $C_n$ denote the subgroup of $\Gamma_c(H_{j,i}^{(k)}, \sfZ)$ consisting of the sections $x$ supported on the free orbits and such that $s x = (-1)^{\absv{s}} x$ for $s \in S_{i + 1}$.
Then $H_\bullet(H_j, F^{i,a})$ is the homology of $C_\bullet$ by the c-softness of $F^{i,a}$.
Hence we want to show
\begin{align*}
H_0(C_\bullet) &\cong D^s(\Sigma_{j, i})_{,A},&
H_k(C_\bullet) &= 0 \quad (k > 0).
\end{align*}

Let $H'_{j, i}$ be the quotient groupoid $H_{j, i} / S_{i+1}$.
On the one hand, $H'_{j, i}$ is AF because $H_{j, i}$, which is Morita equivalent to $R^u(\Sigma_{j, i}, \sigma_{j, i})$, is AF (notice that the quotient by the action of $S_{i + 1}$ does not create any stabilizers).
On the other, $H_\bullet(H_j, F^{i,a})$ can be regarded as the groupoid homology of $H'_{j, i}$ with the induced action of $F^{i,a}$. Then we can conclude that the higher homology groups vanish by virtue of Proposition~\ref{prop:af-grpd-triv-homol}.

Finally, the $0$-th homology is the group of $H_j$-coinvariants of $C_0 = \Gamma_c(T, F^{i,a})$.
By the concrete description of $D^s(\Sigma_{j, i})_{,A}$ given by~\cite{put:HoSmale}*{Theorems 3.3.3 and 4.2.8}, we have $H_0(C_\bullet) \cong D^s(\Sigma_{j, i})_{,A}$.
More precisely, any element of the alternating subgroup of $D^s(\Sigma_{j, i})$ (with respect to the action of $S_{i+1}$) is represented by a linear combination of compact open subsets of a stable set with alternating coefficients.
Such compact open sets must belong to the free part $T_i^f \subset T_i$.
\end{proof}

Combining Theorem~\ref{thm:compar-homologies} with the main results in~\cite{valmak:kpd} we obtain the following results. The first is a Künneth formula for the product of two Smale spaces, a result that generalizes~\cite{valmak:groupoidtwo}*{Theorem 5.2} and~\cite{dkw:dyn}*{Theorem 6.5}. 
\begin{theorem}\label{thm:kunses}
Let $(Y_1, \psi_1)$ and $(Y_2, \psi_2)$ be non-wandering Smale spaces.
Then we have a split exact sequence
\begin{multline*}
0 \to {\bigoplus_{a + b = k}} H^s_a(Y_1, \psi_1) \otimes H^s_b(Y_2, \psi_2) \to H^s_k(Y_1 \times Y_2, \psi_1 \times \psi_2)\\
\quad \to {\bigoplus_{a + b = k-1}} \Tor(H^s_a(Y_1, \psi_1), H^s_b(Y_2, \psi_2)) \to 0.
\end{multline*}
\end{theorem}

\begin{proof}
This is obtained by combining Theorem~\ref{thm:compar-homologies} with a general Künneth type theorem for groupoid homology~\cite{valmak:kpd}*{Theorem A}.
\end{proof}

As usual, the morphisms in the short exact sequence of Theorem~\ref{thm:kunses} above are natural in
any reasonable sense, however the splitting is not.

The second result can be interpreted as a Poincaré duality-type theorem in the setting of Smale spaces. It reduces the computation of the homology groups to the standard (compactly supported) cohomology groups for the topological space underlying the dynamical system.

Suppose that $(X, \phi)$ is a Smale space such that the unstable sets $X^u(x)$ are homeomorphic to $\R^n$, for some fixed $n$.
Then we can consider a local system $\underline{o}$ for the group $\Z/2\Z$ corresponding to the leafwise orientation on unstable sets.
Formally speaking, its stalk at $x \in X$ is given by
\[
\underline{o}_x = (\medwedge^n T_x X^u(x, \epsilon) \setminus \{0\}) / \R_{ > 0}
\]
for a small enough $\epsilon > 0$.

\begin{theorem}\label{thm:poindualsmale}
Let $(X, \phi)$ be as above.
Then we have an isomorphism
\[
H^s_k(X, \phi) \cong H^{n-k}(X, \sfZ \times_{\Z/2\Z} \underline{o}).
\]
\end{theorem}

\begin{proof}
This is obtained by combining Theorem~\ref{thm:compar-homologies} with~\cite{valmak:kpd}*{Theorem B}.
\end{proof}

\begin{remark}
In~\cite{put:HoSmale}*{Question 8.3.2} Putnam asked if his homology groups of $(X, \phi)$ are isomorphic, up to a degree shift, to the usual cohomology of the underlying space $X$ whenever the (un)stable sets are contractible.
The above result gives a precise form of this conjectural isomorphism.
\end{remark}

Let $(X, \phi)$ be a compact Riemannian manifold with an Anosov diffeomorphism.
If the non-wandering set of this system is $X$ itself, we obtain a Smale space which satisfies the assumption of the above theorem, see for example~\cite{bs:dynbook}*{Section 5.6}.

\begin{example}\label{exa:rotalg}
As a concrete example of the above, let us consider the $m$-dimensional torus $\bT^m$ together with the automorphism $\phi_A$ induced by a hyperbolic matrix $A \in \mathrm{SL}_m(\Z)$.
Our convention here is that $A$ is diagonalizable as a real matrix, with eigenvalues $(\lambda_i)_{i = 1}^m$ such that $\absv{\lambda_i} \neq 1$ for all $i$.
Let us assume that they are ordered so that we have $\absv{\lambda_i} < 1$ for $i = 1, \ldots, n$, while $\absv{\lambda_i} > 1$ for $i = n + 1, \ldots, m$, and let us denote an eigenvector for $\lambda_i$ by $v_i$.

Identifying the tangent spaces of $\bT^m$ with $\R^m$, the stable direction of $\phi_A$ is spanned by the vectors $(v_i)_{i = 1}^n$ , while the unstable direction is spanned by the $(v_i)_{i = n + 1}^m$.
As a transversal $T \subset \bT^m$ for $R^u(\bT^m, \phi_A)$, we take the span (based at the origin $0 \in \bT^m$) of a subset $\{ e_1, \dots, e_n \}$ of the standard basis of $\R^m$ such that $\{ e_1, \dots, e_n, v_{n + 1}, \dots, v_m \}$ is again a basis.
Then the action of the étale groupoid $G = R^u(\bT^m, \phi_A)|_T$ on $T$ is given by translation, hence preserves the standard orientation on $T$.

By Theorem~\ref{thm:poindualsmale}, we see that $H^s_k(\bT^m, \phi_A)$ is isomorphic to the ordinary cohomology groups of $\bT^m$ up to degree shift by $n$, that is, we have
\begin{equation}\label{eq:torus-ex}
H^s_{k}(\bT^m, \phi_A) \cong \Z^r, \quad r = \binom{m}{n-k}.
\end{equation}
\end{example}

\section{Smale space with totally disconnected unstable sets}
\label{sec:smale-sp-tot-disc-unst-sets}

Now, let $(X, \phi)$ be a Smale space with totally disconnected unstable sets, and let us fix a $u$-bijective map $f \colon (\Sigma, \sigma) \to (X, \phi)$ from a shift of finite type.
Let us take the constructions from Section~\ref{sec:sof-res-from-u-bij-map}, with $(Z, \zeta) = (\Sigma, \sigma)$.
In particular, we have an étale groupoid $G = R^u(X, \phi)|_T$ for a transversal $T$ as in Proposition~\ref{prop:lift-transv}, and the sheaves $(F^{n,a})_n$ giving a resolution of $\sfZ_T$.

Recall also that the Smale spaces $(\Sigma_n, \sigma_n)$ admit presentation by certain graphs~\cite{put:HoSmale}.
One starts from a graph $\cG$ presenting $(\Sigma, \sigma)$ with vertex set $\cG^0$ and edge set $\cG^1$.
Then one defines graphs $\cG_n$ representing $(\Sigma_n, \sigma_n)$, whose vertex sets are subsets of $(\cG^0)^{n+1}$, and the edge sets are subsets of $(\cG^1)^{n+1}$.
There are natural actions of $S_{n+1}$ on $(\cG^i)^{n+1}$ inducing the one on $\Sigma_n$.

We keep using the notation introduced in Sections~\ref{sec:adapt-segal} and~\ref{sec:sof-res-from-u-bij-map}.

\begin{proposition}
We have
\[
K_{i}(G \ltimes C_0(\nvGR{T_\bullet}^{(p)} \smallsetminus \nvGR{T_\bullet}^{(p-1)})) \cong H_0(G, F^{p,a})
\]
when $i \equiv p \bmod 2$ and $K_{i}(G \ltimes C_0(\nvGR{T_\bullet}^{(p)} \smallsetminus \nvGR{T_\bullet}^{(p-1)})) = 0$ otherwise.
\end{proposition}

\begin{proof}
Put $A = T_p / S_{p + 1}$ and $B = \nvGR{T_\bullet}^{(p)} \smallsetminus \nvGR{T_\bullet}^{(p-1)} \simeq (\mathring{\Delta}^p \times T_p^f) / S_{p + 1}$.
As $A$ is a totally disconnected $G$-space, $G \ltimes A$ is an ample groupoid.
Moreover, the map $B \to A$ induced by the projection of $\mathring{\Delta}^p \times T_p^f$ to the second factor is $G$-equivariant, hence $C_0(B)$ has a structure of $G \ltimes A$-C$^*$-algebra.
By Theorem~\ref{thm:spec-seq-gen} and Remark~\ref{rem:spec-seq-gen-special-cases}, we have a spectral sequence
\[
E^2_{p q} = H_p(G \ltimes A, K_q(C_0(B))) \Rightarrow K_{p + q}(G \ltimes C_0(B)).
\]

Moreover, as in the proof of Theorem~\ref{thm:compar-homologies}, $G \ltimes A$ is an AF groupoid.
By Proposition~\ref{prop:af-grpd-triv-homol} we have $E^2_{p q} = 0$ for $p > 0$ for the above spectral sequence, hence it collapses at the $E^2$-sheet.
We are left with establishing the identification of $H_0(G \ltimes A, K_q(C_0(B)))$ with either $H_0(G, F^{p,a})$ or $0$ depending on the parity of $q$. We will achieve this by comparing the corresponding $G$-modules up to replacing the transversal, namely $K_q(C_0(B))$ and $\Gamma_c(T'', F^{p, a})$ for some $T'' \subset X$, as the $0$-th homology groups are groups of their coinvariants.

By Morita invariance of groupoid homology, we may replace the transversal $T_p \subset \Sigma_p$ by another $S_{p+1}$-invariant transversal.
For each $z \in \cG_p^0 \subset (\cG^0)^{p+1}$, let $A_z$ be the space of one-sided infinite paths representing a local stable set in $\Sigma_p$.
As a new transversal, take  one that can be identified with the union of the $A_z$, and call it $T'$.
Then its free part $T^{\prime f}$ is the subset of paths that pass through a vertex $w \in \cG_p^0$ which is in a free orbit of the $S_{p+1}$ action.

Now, let $B'$ be the quotient $(\mathring{\Delta}^p \times T^{\prime f}) / S_{p + 1}$, which corresponds to $B$ under the Morita equivalence induced by the change of transversal.
Consider the homomorphism
\[
j\colon K_q(C_0(B')) \to K_{q-p}(C_0(T^{\prime f}))
\]
induced by the projection $\mathring{\Delta}^p  \times T^{\prime f} \to B'$ and the Bott periodicity isomorphism
\[
K_q(C_0(\mathring{\Delta}^p  \times T^{\prime f})) \simeq K_{q - p}(C_0(T^{\prime f})).
\]
The right hand side can be identified with $\Gamma_c(T^{\prime f}, \sfZ)$, or $0$, depending on $(q-p) \bmod~2$.
It is enough to check that $j$ is injective, and identify its image with the alternating subspace for the action of $S_{p + 1}$. To do this, we look at the structure of the $S_{p+1}$-C$^*$-algebra $C_0(T^{\prime f})$.

By the above presentation of $T^{\prime f}$, $C_0(T^{\prime f})$ is the inductive limit of an increasing sequence of $S_{p+1}$-invariant finite-dimensional subalgebras of the form $C(K_n)$, where $K_n$ is a free $S_{p+1}$-set.
Concretely, $K_n$ is the set of paths of length $n$ in $\cG_p$ passing through a free orbit.

Let $B_n$ be the quotient $(\mathring{\Delta}^p \times K_n) / S_{p + 1}$.
Then $C_0(B)$ is again the union of $C(B_n)$, which implies $K_q(C_0(B)) = \varinjlim K_q(C(B_n))$.
It remains to show that $j$ is injective on $K_q(C_0(B_n))$ and its image is the alternating part of $\Gamma_c(K_n, \sfZ) = C(K_n, \Z)$.
Since $K_n$ is a finite free $S_{p + 1}$-set, we can replace $K_n$ by $S_{p + 1}$.

Now, $(\mathring{\Delta}^p \times S_{p + 1}) / S_{p + 1} \simeq \mathring{\Delta}^p$ implies
\begin{equation}\label{eq:K-grp-Dp-S-p-plus-1-mod-S-p-plus-1}
K_p(C((\mathring{\Delta}^p \times S_{p + 1}) / S_{p + 1})) \simeq \Z.
\end{equation}
Looking at the quotient map $\mathring{\Delta}^p \times S_{p + 1} \to (\mathring{\Delta}^p \times S_{p + 1}) / S_{p + 1}$, we see that the induced homomorphism
\[
K_p(C((\mathring{\Delta}^p \times S_{p + 1}) / S_{p + 1})) \to K_p(C(\mathring{\Delta}^p \times S_{p + 1})) \simeq C(S_{p + 1}, \Z)
\]
sends a generator of~\eqref{eq:K-grp-Dp-S-p-plus-1-mod-S-p-plus-1} to the element $\sum_{s \in S_{p + 1}} (-1)^{\absv{s}} s$.
As this element is a generator of the alternating part of $C(S_{p + 1}, \Z)$, we obtain the claim.
\end{proof}

We now have a dual analogue of the homological spectral sequence in Theorem~\ref{thm:spec-seq-Smale-tot-discon-st-set}.

\begin{corollary}\label{cor:spec-seq-Smale-tot-disconn-stable-sets}
Let $(X, \phi)$ be as above.
There is a cohomological spectral sequence abutting to $K_\bullet(C^* R^u(X, \phi))$, with $E_2$-sheet given by $H^s_\bullet(X, \phi)$.
More precisely:
\[
E_2^{pq}= H^s_p(X,\phi)\otimes K_q(\bC)\Rightarrow K_{p+q}(C^* R^u(X, \phi)).
\]
\end{corollary}

\begin{proof}
The $E_2$-sheet of the spectral sequence \eqref{eq:spec-seq-for-K-of-G-crossed} is given by the cohomology of $H_0(G, F^{p,a})$ for even $q$, and is trivial for odd $q$.

We first claim that, up to inverting the degree, this cohomology is equal to $H_\bullet(G, \sfZ)$.
Indeed, as $(F^{p,a})_p$ is a resolution of $\sfZ$, the groupoid homology is computed from the double complex with terms $\Gamma_c(G^{(q)}, F^{p,a})$.
When we fix $p$, the resulting homological complex has homology groups
\[
H_q(G, F^{p,a}) \cong H_q(G \ltimes A, F')
\]
for $A = T_p/S_{p+1}$ as before, and the $G \ltimes A$-sheaf $F'$ corresponding to $F^{p,a}$.
Again by Proposition~\ref{prop:af-grpd-triv-homol}, this homology is trivial for $q > 0$.
Then the homology of the total complex is the same as the cohomology of the cochain complex $H_0(G, F^{\bullet,a})$.

Thus, we have
\begin{align*}
E_2^{p q} &\cong H_{-p}(G, \sfZ) \quad (\text{$q$ even}), &
E_2^{p q} &\cong 0 \quad (\text{$q$ odd}),
\end{align*}
which implies the claim.
\end{proof}

This corollary, combined with the dual result from~\cite{valmak:groupoidtwo} (summarized in Theorem~\ref{thm:spec-seq-Smale-tot-discon-st-set} above), provides a partial solution to a conjecture in~\cite{put:HoSmale}*{Question 8.4.1}, asking about the possibility of computing $K$-groups of C$^*$-algebras attached to Smale spaces from the corresponding homology groups. We can consider the conjecture settled in the case where either the stable or unstable foliations are totally disconnected.

\section{Finiteness for general Smale spaces}
\label{sec:finite-gen-smale-sp}

Although we cannot directly relate the $K$-groups of $C^* R^u(X, \phi)$ to the groupoid homology of $R^u(X, \phi)|_T$, the considerations from Section~\ref{sec:symm-simp-sp} still lead to the following finiteness result.

\begin{theorem}\label{thm:ftrk}
Let $(X, \phi)$ be a non-wandering Smale space.
Then $K_\bullet(C^* R^u(X, \phi))$ is of finite rank.
\end{theorem}

\begin{proof}
Let us fix a $u$-bijective map $g \colon (Z, \zeta) \to (X, \phi)$ from another Smale space with totally disconnected stable sets, with $g$ at most $N$-to-one.
Take a transversal $T \subset X$ as in Proposition~\ref{prop:lift-transv}, and consider the étale groupoid $G = R^u(X, \phi)|_T$ and the totally disconnected $G$-space $T_0 = g^{-1}(T)$.
Let us take the constructions from Sections~\ref{sec:equiv-k-th-spec-seq} and~\ref{sec:sof-res-from-u-bij-map}.
Then the rationalization of the cohomological spectral sequence \eqref{eq:spec-seq-for-K-of-G-crossed} becomes
\begin{equation}\label{eq:rat-cohom-sp-seq}
E_1^{p q} = K_{p+q}(G \ltimes C_0(\nvGR{T_\bullet}^{(p)} \setminus \nvGR{T_\bullet}^{(p-1)})) \otimes \Q \Rightarrow K_{p + q}(C^*_r G) \otimes \Q.
\end{equation}

Let us fix $p$.
Again by Theorem~\ref{thm:spec-seq-gen}, there is a homological spectral sequence
\[
E^1_{p' q'} = K_{q'}(C_0(G^{(p')} \ltimes (\nvGR{T_\bullet}^{(p)} \setminus \nvGR{T_\bullet}^{(p-1)}))) \Rightarrow  K_{p'+q'}(G \ltimes C_0(\nvGR{T_\bullet}^{(p)} \setminus \nvGR{T_\bullet}^{(p-1)}))
\]
compatible with the actions of $S_{p+1}$.
By Proposition~\ref{prop:rat-segal-e1-sheet}, rationalization at the $E^1$-sheet agrees with
\[
(K_{q' + p}(C_0(G^{(p')} \ltimes T_p)) \otimes \Q)_A,
\]
which converges to $(K_q(C^*(G \ltimes T_p)) \otimes \Q)_A$.
This implies that the spectral sequence \eqref{eq:rat-cohom-sp-seq} satisfies
\[
E_1^{p q} \cong (K_q(C^*(G \ltimes T_p)) \otimes \Q)_A.
\]
The right hand side can be considered as a submodule of $K_q(C^*(R^u(Z_p, \zeta_p))) \otimes \Q$.
On one hand, we know that $K_q(C^*(R^u(Z_p, \zeta_p)))$ is of finite rank by~\cite{valmak:groupoidtwo}*{Corollary 3.10}.
On the other, we have $E_1^{p q} = 0$ for $p > N$, hence the assertion.
\end{proof}

Recall that taking the inverse homeomorphism we get the stable equivalence relation:
\[
R^s(X, \phi) = R^u(X, \phi^{-1}).
\]
In particular, $K_\bullet(C^* R^s(X, \phi))$ is also of finite rank by the above theorem.

\begin{remark}\label{rem:kpwiso}
In~\cite{MR3692021}, Kaminker, Putnam, and Whittaker considered $K$-theoretic duality for the \emph{Ruelle algebras}
\begin{align*}
\cR_s(X, \phi) &= C^*(R^s(X, \phi)) \rtimes_\phi \Z,&
\cR_u(X, \phi) &= C^*(R^u(X, \phi)) \rtimes_\phi \Z.
\end{align*}
They showed that these algebras are odd Spanier--Whitehead dual to each other~\cite{val:shi}, and that this duality implies $\cR_s(X, \phi) \cong \cR_u(X, \phi)$ under the assumption that $K_\bullet(C^* R^s(X, \phi))$ and $K_\bullet(C^* R^u(X, \phi))$ are of finite rank.
They have also conjectured that this assumption is unnecessary, and our result above implies that this is indeed the case.
\end{remark}

\begin{remark}
In view of the classification program of C$^*$-algebras and the results in~\citelist{\cite{deeley:strung}\cite{dgy:rrzero}}, it is natural to ask whether the range of $K$-theory on the class of (simple) classifiable, real rank zero C$^*$-algebras is exhausted by the (un)stable algebras of Smale spaces up to restricting to a transversal and passing to hereditary subalgebras.

However, the above finiteness implies that this is not possible.
Indeed, there are algebras as above whose $K$-groups are not of finite rank.
For example, the simple AF algebra associated to the following graph
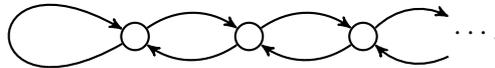
\[
\begin{tikzpicture}[->,>=stealth',auto,node distance=1.5cm,
                    thick,main node/.style={circle,draw},
					every loop/.style={min distance=10mm}]
\node[main node] (1) {};
\node[main node] (2) [right of= 1] {};
\node[main node] (3) [right of= 2] {};
\node (4) [right of= 3] {$\cdots$,};
\path[]
  (1)   edge [loop left,in=150,out=210,looseness=30] node [yshift=2pt] {} (1)
  (1)   edge [bend left=35] node {} (2)
  (2)   edge [bend left=35] node {} (1)
  (2)   edge [bend left=35] node {} (3)
  (3)   edge [bend left=35] node {} (2)
  (3)   edge [bend left=35] node {} (4)
  (4)   edge [bend left=35] node {} (3) ;
\end{tikzpicture}
\]
or $A^{\otimes \infty}$ for a classifiable algebra $A$ whose $K$-groups have rank more than $1$ will do.
\end{remark}


\raggedright


\begin{bibdiv}
\begin{biblist}

\bib{MR1631708}{article}{
      author={Anderson, Jared~E.},
      author={Putnam, Ian~F.},
       title={Topological invariants for substitution tilings and their
  associated {$C^*$}-algebras},
        date={1998},
        ISSN={0143-3857},
     journal={Ergodic Theory Dynam. Systems},
      volume={18},
      number={3},
       pages={509\ndash 537},
         url={http://dx.doi.org/10.1017/S0143385798100457},
         doi={10.1017/S0143385798100457},
      review={\MR{1631708}},
}


\bib{bdgw:matui}{article}{
author = {Bönicke, Christian},
author = {Dell'Aiera, Clément},
author = {Gabe, James},
author = {Willett, Rufus},
title = {Dynamic asymptotic dimension and Matui's HK conjecture},
journal = {Proceedings of the London Mathematical Society},
volume = {126},
number = {4},
pages = {1182-1253},
doi = {https://doi.org/10.1112/plms.12510},
url = {https://londmathsoc.onlinelibrary.wiley.com/doi/abs/10.1112/plms.12510},
eprint = {https://londmathsoc.onlinelibrary.wiley.com/doi/pdf/10.1112/plms.12510},
year = {2023}
}

\bib{val:kthpgrp}{article}{
      author={B\"{o}nicke, Christian},
      author={Proietti, Valerio},
      title={Categorical approach to the {B}aum--{C}onnes conjecture for
  \'{e}tale groupoids},
      date={2024-01-02},
      journal={Journal of the Institute of Mathematics of Jussieu},
      url={https://www.cambridge.org/core/journals/journal-of-the-institute-of-mathematics-of-jussieu/article/categorical-approach-to-the-baumconnes-conjecture-for-etale-groupoids/809AEC9077F4253A7D75BA472D26D2CD},
      doi={10.1017/S1474748023000531},
      note= {Online First},
}

\bib{bs:dynbook}{book}{
      author={Brin, Michael},
      author={Stuck, Garrett},
       title={Introduction to dynamical systems},
   publisher={Cambridge University Press, Cambridge},
        date={2002},
        ISBN={0-521-80841-3},
         url={https://doi.org/10.1017/CBO9780511755316},
         doi={10.1017/CBO9780511755316},
      review={\MR{1963683}},
}

\bib{cramo:hom}{article}{
      author={Crainic, Marius},
      author={Moerdijk, Ieke},
       title={A homology theory for {\'e}tale groupoids},
        date={2000},
        ISSN={0075-4102},
     journal={J. Reine Angew. Math.},
      volume={521},
       pages={25\ndash 46},
  eprint={\href{http://arxiv.org/abs/math/9905011}{\texttt{arXiv:math/9905011
  [math.KT]}}},
         url={http://dx.doi.org/10.1515/crll.2000.029},
         doi={10.1515/crll.2000.029},
      review={\MR{1752294 (2001f:58039)}},
}

\bib{deeley:hk}{article}{
      author={Deeley, Robin~J.},
       title={A counterexample to the {HK}-conjecture that is principal},
        date={2023},
        ISSN={0143-3857},
     journal={Ergodic Theory Dynam. Systems},
      volume={43},
      number={6},
       pages={1829\ndash 1846},
      eprint={\href{http://arxiv.org/abs/2106.01527}{\texttt{arXiv:2106.01527
  [math.KT]}}},
         url={https://doi.org/10.1017/etds.2022.25},
         doi={10.1017/etds.2022.25},
      review={\MR{4583796}},
}

\bib{dgy:rrzero}{article}{
      author={Deeley, Robin~J.},
      author={Goffeng, Magnus},
      author={Yashinski, Allan},
       title={Smale space {$C^*$}-algebras have nonzero projections},
        date={2020},
        ISSN={0002-9939},
     journal={Proc. Amer. Math. Soc.},
      volume={148},
      number={4},
       pages={1625\ndash 1639},
      eprint={\href{http://arxiv.org/abs/1901.10324}{\texttt{arXiv:1901.10324
  [math.OA]}}},
         url={https://doi.org/10.1090/proc/14837},
         doi={10.1090/proc/14837},
      review={\MR{4069199}},
}

\bib{dkw:dyn}{article}{
      author={Deeley, Robin~J.},
      author={Killough, D.~Brady},
      author={Whittaker, Michael~F.},
       title={Dynamical correspondences for {S}male spaces},
        date={2016},
        ISSN={1076-9803},
     journal={New York J. Math.},
      volume={22},
       pages={943\ndash 988},
      eprint={\href{http://arxiv.org/abs/1505.05558}{\texttt{arXiv:1505.05558
  [math.DS]}}},
         url={http://nyjm.albany.edu:8000/j/2016/22_943.html},
      review={\MR{3576278}},
}

\bib{dkw:func}{article}{
      author={Deeley, Robin~J.},
      author={Killough, D.~Brady},
      author={Whittaker, Michael~F.},
       title={Functorial properties of {P}utnam's homology theory for {S}male
  spaces},
        date={2016},
        ISSN={0143-3857},
     journal={Ergodic Theory Dynam. Systems},
      volume={36},
      number={5},
       pages={1411\ndash 1440},
      eprint={\href{http://arxiv.org/abs/1407.2992}{\texttt{arXiv:1407.2992
  [math.DS]}}},
         url={https://doi.org/10.1017/etds.2014.134},
      review={\MR{3519418}},
}

\bib{deeley:strung}{article}{
      author={Deeley, Robin~J.},
      author={Strung, Karen~R.},
       title={Nuclear dimension and classification of {$\rm C^*$}-algebras
  associated to {S}male spaces},
        date={2018},
        ISSN={0002-9947},
     journal={Trans. Amer. Math. Soc.},
      volume={370},
      number={5},
       pages={3467\ndash 3485},
      eprint={\href{http://arxiv.org/abs/1601.02432}{\texttt{arXiv:1601.02432
  [math.OA]}}},
         url={https://doi.org/10.1090/tran/7046},
         doi={10.1090/tran/7046},
      review={\MR{3766855}},
}

\bib{MR2820377}{article}{
      author={Echterhoff, Siegfried},
      author={Emerson, Heath},
       title={Structure and {$K$}-theory of crossed products by proper
  actions},
        date={2011},
        ISSN={0723-0869},
     journal={Expo. Math.},
      volume={29},
      number={3},
       pages={300\ndash 344},
      eprint={\href{http://arxiv.org/abs/1012.5214}{\texttt{arXiv:1012.5214
  [math.KT]}}},
         url={https://doi.org/10.1016/j.exmath.2011.05.001},
         doi={10.1016/j.exmath.2011.05.001},
      review={\MR{2820377}},
}

\bib{eilstee:algtop}{book}{
      author={Eilenberg, Samuel},
      author={Steenrod, Norman},
       title={Foundations of algebraic topology},
   publisher={Princeton University Press, Princeton, New Jersey},
        date={1952},
      review={\MR{0050886}},
}

\bib{MR0482697}{book}{
      author={Engelking, Ryszard},
       title={Dimension theory},
      series={North-Holland Mathematical Library},
   publisher={North-Holland Publishing Co., Amsterdam-Oxford-New York;
  PWN---Polish Scientific Publishers, Warsaw},
        date={1978},
      volume={19},
        ISBN={0-444-85176-3},
         url={https://mathscinet.ams.org/mathscinet-getitem?mr=0482697},
        note={Translated from the Polish and revised by the author},
      review={\MR{0482697}},
}

\bib{simsfarsi:hk}{article}{
      author={Farsi, Carla},
      author={Kumjian, Alex},
      author={Pask, David},
      author={Sims, Aidan},
       title={Ample groupoids: equivalence, homology, and {M}atui's {HK}
  conjecture},
        date={2019},
        ISSN={1867-5778},
     journal={M\"{u}nster J. Math.},
      volume={12},
      number={2},
       pages={411\ndash 451},
      eprint={\href{http://arxiv.org/abs/1808.07807}{\texttt{arXiv:1808.07807
  [math.OA]}}},
         url={https://doi-org.ezproxy.uio.no/10.17879/53149724091},
         doi={10.17879/53149724091},
      review={\MR{4030921}},
}

\bib{MR998125}{article}{
      author={Fiedorowicz, Zbigniew},
      author={Loday, Jean-Louis},
       title={Crossed simplicial groups and their associated homology},
        date={1991},
        ISSN={0002-9947},
     journal={Trans. Amer. Math. Soc.},
      volume={326},
      number={1},
       pages={57\ndash 87},
         url={http://dx.doi.org/10.2307/2001855},
         doi={10.2307/2001855},
      review={\MR{998125 (91j:18018)}},
}

\bib{MR0345092}{book}{
      author={Godement, Roger},
       title={Topologie alg\'{e}brique et th\'{e}orie des faisceaux},
   publisher={Hermann, Paris},
        date={1973},
        note={Troisi\`eme \'{e}dition revue et corrig\'{e}e, Publications de
  l'Institut de Math\'{e}matique de l'Universit\'{e} de Strasbourg, XIII,
  Actualit\'{e}s Scientifiques et Industrielles, No. 1252},
      review={\MR{0345092}},
}

\bib{MR2840650}{book}{
      author={Goerss, Paul~G.},
      author={Jardine, John~F.},
       title={Simplicial homotopy theory},
      series={Modern Birkh\"auser Classics},
   publisher={Birkh\"auser Verlag, Basel},
        date={2009},
        ISBN={978-3-0346-0188-7},
         url={https://doi.org/10.1007/978-3-0346-0189-4},
        note={Reprint of the 1999 edition [MR1711612]},
      review={\MR{2840650}},
}

\bib{MR3692021}{article}{
      author={Kaminker, Jerome},
      author={Putnam, Ian~F.},
      author={Whittaker, Michael~F.},
       title={K-theoretic duality for hyperbolic dynamical systems},
        date={2017},
        ISSN={0075-4102},
     journal={J. Reine Angew. Math.},
      volume={730},
       pages={263\ndash 299},
      eprint={\href{http://arxiv.org/abs/1009.4999}{\texttt{arXiv:1009.4999
  [math.KT]}}},
         url={https://doi.org/10.1515/crelle-2014-0126},
      review={\MR{3692021}},
}

\bib{kp:clas}{article}{
      author={Kirchberg, Eberhard},
      author={Phillips, N.~Christopher},
       title={Embedding of exact {$C^*$}-algebras in the {C}untz algebra {$\scr
  O_2$}},
        date={2000},
        ISSN={0075-4102,1435-5345},
     journal={J. Reine Angew. Math.},
      volume={525},
       pages={17\ndash 53},
         url={https://doi.org/10.1515/crll.2000.065},
         doi={10.1515/crll.2000.065},
      review={\MR{1780426}},
}

\bib{krieger:inv}{article}{
      author={Krieger, Wolfgang},
       title={On dimension functions and topological {M}arkov chains},
        date={1980},
        ISSN={0020-9910},
     journal={Invent. Math.},
      volume={56},
      number={3},
       pages={239\ndash 250},
         url={https://doi.org/10.1007/BF01390047},
         doi={10.1007/BF01390047},
      review={\MR{561973}},
}

\bib{gall:kk}{article}{
      author={Le~Gall, Pierre-Yves},
       title={Th\'eorie de {K}asparov \'equivariante et groupo\"\i des. {I}},
        date={1999},
        ISSN={0920-3036},
     journal={$K$-Theory},
      volume={16},
      number={4},
       pages={361\ndash 390},
         url={https://doi.org/10.1023/A:1007707525423},
         doi={10.1023/A:1007707525423},
      review={\MR{1686846}},
}

\bib{mats:ruellemarkov}{article}{
      author={Matsumoto, Kengo},
       title={Topological conjugacy of topological {M}arkov shifts and {R}uelle
  algebras},
        date={2019},
        ISSN={0379-4024},
     journal={J. Operator Theory},
      volume={82},
      number={2},
       pages={253\ndash 284},
      eprint={\href{http://arxiv.org/abs/1706.07155}{\texttt{arXiv:1706.07155
  [math.OA]}}},
      review={\MR{4015953}},
}

\bib{MR2876963}{article}{
      author={Matui, Hiroki},
       title={Homology and topological full groups of \'etale groupoids on
  totally disconnected spaces},
        date={2012},
        ISSN={0024-6115},
     journal={Proc. Lond. Math. Soc. (3)},
      volume={104},
      number={1},
       pages={27\ndash 56},
      eprint={\href{http://arxiv.org/abs/0909.1624}{\texttt{arXiv:0909.1624
  [math.OA]}}},
         url={https://doi.org/10.1112/plms/pdr029},
      review={\MR{2876963}},
}

\bib{MR3552533}{article}{
      author={Matui, Hiroki},
       title={\'{E}tale groupoids arising from products of shifts of finite
  type},
        date={2016},
        ISSN={0001-8708},
     journal={Adv. Math.},
      volume={303},
       pages={502\ndash 548},
      eprint={\href{http://arxiv.org/abs/1512.01724}{\texttt{arXiv:1512.01724
  [math.OA]}}},
         url={https://doi.org/10.1016/j.aim.2016.08.023},
         doi={10.1016/j.aim.2016.08.023},
      review={\MR{3552533}},
}

\bib{meyer:tri}{article}{
      author={Meyer, Ralf},
       title={Homological algebra in bivariant {$K$}-theory and other
  triangulated categories. {II}},
        date={2008},
        ISSN={1875-158X},
     journal={Tbil. Math. J.},
      volume={1},
       pages={165\ndash 210},
      eprint={\href{http://arxiv.org/abs/0801.1344}{\texttt{arXiv:0801.1344
  [math.KT]}}},
      review={\MR{2563811}},
}

\bib{meyernest:tri}{incollection}{
      author={Meyer, Ralf},
      author={Nest, Ryszard},
       title={Homological algebra in bivariant {$K$}-theory and other
  triangulated categories. {I}},
        date={2010},
   booktitle={Triangulated categories},
      series={London Math. Soc. Lecture Note Ser.},
      volume={375},
   publisher={Cambridge Univ. Press, Cambridge},
       pages={236\ndash 289},
  eprint={\href{http://arxiv.org/abs/math/0702146}{\texttt{arXiv:math/0702146
  [math.KT]}}},
      review={\MR{2681710}},
}

\bib{murewi:morita}{article}{
      author={Muhly, Paul~S.},
      author={Renault, Jean~N.},
      author={Williams, Dana~P.},
       title={Equivalence and isomorphism for groupoid {$C^\ast$}-algebras},
        date={1987},
        ISSN={0379-4024},
     journal={J. Operator Theory},
      volume={17},
      number={1},
       pages={3\ndash 22},
      review={\MR{873460 (88h:46123)}},
}


\bib{MR2162164}{book}{
      author={Nekrashevych, Volodymyr},
       title={Self-similar groups},
      series={Mathematical Surveys and Monographs},
   publisher={American Mathematical Society, Providence, RI},
        date={2005},
      volume={117},
        ISBN={0-8218-3831-8},
         url={https://doi.org/10.1090/surv/117},
         doi={10.1090/surv/117},
      review={\MR{2162164}},
}

\bib{MR2526786}{article}{
      author={Nekrashevych, Volodymyr},
       title={{$C^*$}-algebras and self-similar groups},
        date={2009},
        ISSN={0075-4102},
     journal={J. Reine Angew. Math.},
      volume={630},
       pages={59\ndash 123},
         url={https://mathscinet.ams.org/mathscinet-getitem?mr=2526786},
         doi={10.1515/CRELLE.2009.035},
      review={\MR{2526786}},
}

\bib{val:shi}{article}{
    author = {Nishikawa, S.},
    author = {Proietti, V.},
    title     = {Groups with Spanier--Whitehead duality},
    journal = {Annals of {$K$}-theory},
    year      = {2020},
    volume = {5},
    number = {3},
    pages = {465--500},
    DOI = {10.2140/akt.2020.5.465},
}

\bib{phil:clas}{article}{
      author={Phillips, N.~C.},
       title={A classification theorem for nuclear purely infinite simple {$C^*$}-algebras},
        date={2000},
     journal={Doc. Math.},
      volume={5},
       pages={49-114},
}

\bib{val:smale}{article}{
      author={Proietti, Valerio},
       title={A note on homology for {S}male spaces},
        date={2020},
        ISSN={1661-7207},
     journal={Groups Geom. Dyn.},
      volume={14},
      number={3},
       pages={813\ndash 836},
      eprint={\href{http://arxiv.org/abs/1807.06922}{\texttt{arXiv:1807.06922
  [math.KT]}}},
         url={https://doi-org.ezproxy.uio.no/10.4171/ggd/564},
         doi={10.4171/ggd/564},
      review={\MR{4167022}},
}

\bib{valmak:groupoid}{article}{
   author={Proietti, Valerio},
   author={Yamashita, Makoto},
   title={Homology and $K$-theory of dynamical systems I. Torsion-free ample
   groupoids},
   journal={Ergodic Theory Dynam. Systems},
   volume={42},
   date={2022},
   number={8},
   pages={2630--2660},
   issn={0143-3857},
   review={\MR{4448401}},
   doi={10.1017/etds.2021.50},
   eprint={\href{http://arxiv.org/abs/2006.08028}{\texttt{arXiv:2006.08028 [math.KT]}}},
}

\bib{valmak:groupoidtwo}{article}{
   author={Proietti, Valerio},
   author={Yamashita, Makoto},
   title={Homology and $K$-theory of dynamical systems II. Smale spaces with
   totally disconnected transversal},
   journal={J. Noncommut. Geom.},
   volume={17},
   date={2023},
   number={3},
   pages={957--998},
   issn={1661-6952},
   review={\MR{4626307}},
   doi={10.4171/jncg/494},
   eprint={\href{http://arxiv.org/abs/2104.10938}{\texttt{arXiv:2104.10938 [math.KT]}}},
}

\bib{valmak:kpd}{article}{
      title={Homology and K-theory of dynamical systems IV. Further structural results on groupoid homology}, 
      DOI={10.1017/etds.2024.37}, 
      journal={Ergodic Theory and Dynamical Systems}, 
      author={Proietti, V.},
      author={Yamashita, M.}, 
      date={2024-05-15}, 
      note={Online First},
}

\bib{put:algSmale}{article}{
      author={Putnam, Ian~F.},
       title={{$C^*$}-algebras from {S}male spaces},
        date={1996},
        ISSN={0008-414X},
     journal={Canad. J. Math.},
      volume={48},
      number={1},
       pages={175\ndash 195},
         url={https://doi.org/10.4153/CJM-1996-008-2},
         doi={10.4153/CJM-1996-008-2},
      review={\MR{1382481}},
}

\bib{put:funct}{article}{
      author={Putnam, Ian~F.},
       title={Functoriality of the {$C^*$}-algebras associated with hyperbolic
  dynamical systems},
        date={2000},
        ISSN={0024-6107},
     journal={J. London Math. Soc. (2)},
      volume={62},
      number={3},
       pages={873\ndash 884},
         url={https://doi.org/10.1112/S002461070000140X},
         doi={10.1112/S002461070000140X},
      review={\MR{1794291}},
}

\bib{put:HoSmale}{article}{
      author={Putnam, Ian~F.},
       title={A homology theory for {S}male spaces},
        date={2014},
        ISSN={0065-9266},
     journal={Mem. Amer. Math. Soc.},
      volume={232},
      number={1094},
       pages={viii+122},
         url={https://doi.org/10.1090/memo/1094},
      review={\MR{3243636}},
}

\bib{put:notes}{misc}{
      author={Putnam, Ian~F.},
       title={Lecture notes on smale spaces},
  how={lecture note},
  url={http://www.math.uvic.ca/faculty/putnam/ln/Smale_spaces.pdf},
        date={2015},
        note={Accessed: 2010-09-30},
}

\bib{put:spiel}{article}{
      author={Putnam, Ian~F.},
      author={Spielberg, Jack},
       title={The structure of {$C^*$}-algebras associated with hyperbolic
  dynamical systems},
        date={1999},
        ISSN={0022-1236},
     journal={Journal of Functional Analysis},
      volume={163},
      number={2},
       pages={279 \ndash  299},
  url={http://www.sciencedirect.com/science/article/pii/S0022123698933791},
         doi={https://doi.org/10.1006/jfan.1998.3379},
}

\bib{MR584266}{book}{
      author={Renault, Jean},
       title={A groupoid approach to {$C^{\ast} $}-algebras},
      series={Lecture Notes in Mathematics},
   publisher={Springer, Berlin},
        date={1980},
      volume={793},
        ISBN={3-540-09977-8},
      review={\MR{584266}},
}

\bib{ruelle:thermo}{book}{
      author={Ruelle, David},
       title={Thermodynamic formalism},
     edition={Second},
      series={Cambridge Mathematical Library},
   publisher={Cambridge University Press, Cambridge},
        date={2004},
        ISBN={0-521-54649-4},
         url={https://doi.org/10.1017/CBO9780511617546},
         doi={10.1017/CBO9780511617546},
        note={The mathematical structures of equilibrium statistical
  mechanics},
      review={\MR{2129258}},
}

\bib{sga4-vbis}{incollection}{
      author={Saint-Donat, B.},
       title={Techniques de descente cohomologique},
        date={1972},
   booktitle={Th\'{e}orie des topos et cohomologie \'{e}tale des sch\'{e}mas.
  tome 2.},
   publisher={Springer-Verlag},
        note={SGA4, Vbis},
}

\bib{scarparo:hk}{article}{
      author={Scarparo, Eduardo},
       title={Homology of odometers},
        date={2020},
        ISSN={0143-3857},
     journal={Ergodic Theory Dynam. Systems},
      volume={40},
      number={9},
       pages={2541\ndash 2551},
      eprint={\href{http://arxiv.org/abs/1811.05795}{\texttt{arXiv:1811.05795
  [math.OA]}}},
         url={https://doi.org/10.1017/etds.2019.13},
         doi={10.1017/etds.2019.13},
      review={\MR{4130816}},
}

\bib{MR232393}{article}{
      author={Segal, Graeme},
       title={Classifying spaces and spectral sequences},
        date={1968},
        ISSN={0073-8301},
     journal={Inst. Hautes \'{E}tudes Sci. Publ. Math.},
      number={34},
       pages={105\ndash 112},
         url={http://www.numdam.org/item?id=PMIHES_1968__34__105_0},
      review={\MR{232393}},
}

\bib{thomsen:smale}{article}{
      author={Thomsen, Klaus},
       title={{C$^*$}-algebras of homoclinic and heteroclinic structure in
  expansive dynamics},
        date={2010},
        ISSN={0065-9266},
     journal={Mem. Amer. Math. Soc.},
      volume={206},
      number={970},
       pages={x+122},
         url={https://doi.org/10.1090/S0065-9266-10-00581-8},
         doi={10.1090/S0065-9266-10-00581-8},
      review={\MR{2667385}},
}

\bib{MR1703305}{article}{
      author={Tu, Jean-Louis},
       title={La conjecture de {B}aum-{C}onnes pour les feuilletages
  moyennables},
        date={1999},
        ISSN={0920-3036},
     journal={$K$-Theory},
      volume={17},
      number={3},
       pages={215\ndash 264},
         url={http://dx.doi.org/10.1023/A:1007744304422},
         doi={10.1023/A:1007744304422},
      review={\MR{1703305 (2000g:19004)}},
}


\end{biblist}
\end{bibdiv}
\end{document}